\documentclass[11pt]{article}

\usepackage[a4paper,margin=1in]{geometry}
\usepackage{amsmath,amssymb,amsthm,bm}
\usepackage{booktabs}
\usepackage[title]{appendix}%
\usepackage{graphicx}
\usepackage{hyperref}
\usepackage[numbers,sort&compress]{natbib}

\hypersetup{
  colorlinks=true,
  linkcolor=blue,
  citecolor=blue,
  urlcolor=blue
}
\newtheorem{lemma}{Lemma}

\newtheorem{theorem}{Theorem}
\newtheorem{assumption}{Assumption}
\newtheorem{remark}{Remark}
\newcounter{algorithm}

\newcommand{\OmegaD}{\Omega}
\newcommand{\curl}{\nabla \times}
\newcommand{\diver}{\nabla \cdot}
\newcommand{\bn}{\bm{\mathrm{n}}}
\newcommand{\E}{\bm{\mathrm{E}}}
\newcommand{\HH}{\bm{\mathrm{H}}}
\newcommand{\uu}{\bm{\mathrm{u}}}
\newcommand{\hh}{\bm{\mathrm{h}}}
\newcommand{\ff}{\bm{\mathrm{f}}}
\newcommand{\DD}{\bm{\mathrm{D}}}
\newcommand{\BB}{\bm{\mathrm{B}}}
\newcommand{\vv}{\bm{\mathrm{v}}}
\newcommand{\dt}{\tau}
\newcommand{\Eh}{\bm{\mathbb E}_h}
\newcommand{\Hh}{\bm{\mathbb H}_h}
\newcommand{\Uh}{\bm{\mathbb U}_h}
\newcommand{\Wh}{\mathbb W_h}
\newcommand{\Ph}{\mathbb P_h}

\newcommand{\Ec}{\bm{\mathbb E}}
\newcommand{\Hc}{\bm{\mathbb H}}
\newcommand{\Uc}{\bm{\mathbb U}}
\newcommand{\Wc}{\mathbb W}
\newcommand{\Pc}{\mathbb P}

\newcommand{\lnorm}[1]{\| #1 \|_{L^2}}
\newcommand{\hnorm}[1]{\| #1 \|_{H^1}}
\newcommand{\ip}[2]{(#1,#2)}

\title{Iterative Decoupling Methods for a Total-Pressure Formulation of Quasi-Static Electroporoelasticity}
\author{
Huipeng Gu\\
{\small Shenzhen University of Information Technology, Shenzhen, Guangdong, P.R. China}\\[0.6em]
Mingchao Cai\\
{\small Department of Mathematics, Morgan State University, Baltimore, Maryland, USA}\\[0.6em]
Jingzhi Li\\
{\small Department of Mathematics, Southern University of Science and Technology,}\\
{\small Shenzhen, Guangdong, P.R. China}
}
\date{}

\begin{document}
\maketitle

\begin{abstract}
Quasi-static electroporoelasticity couples Maxwell's equations with Biot's poroelasticity through electrokinetic coupling between the electric field and the pressure gradient.
By introducing the total pressure, the electroporoelasticity equations are reformulated as a five-field system to address poroelastic locking in the nearly incompressible regime.
For the resulting five-field system, a monolithic weak formulation is derived, together with a continuous stability estimate.
A second-order backward differentiation formula (BDF2) time discretization and a mixed finite- element spatial discretization are then introduced, yieling to a fully discrete monolithic scheme.
Building on this scheme, we develop an iterative decoupling method that alternates between an electromagnetic subproblem and a poroelastic subproblem, and prove its geometric convergence to the monolithic solution with an explicit mesh-independent contraction factor under the physical coupling condition.
An algebraically equivalent reduced form is also presented, in which the electromagnetic block is solved only once per time step, while the poroelastic block is solved iteratively with electric-field correction updates obtained from the pressure-gradient feedback.
Numerical experiments verify the theoretical predictions and demonstrate the locking-free performance. 
\end{abstract}


\section{Introduction}

Electroporoelasticity describes the multiphysics process in which fluid flow caused by solid deformation interacts with electric and magnetic fields through electrokinetic coupling in porous media. Its poromechanical component originates from Biot's theory of poroelasticity \cite{biot1962mechanics}, while Pride's macroscopic equations provide a unified description of electromagnetic and mechanical coupling in porous media \cite{pride1994governing}.
Early numerical studies mainly focused on fully dynamic electroseismic models, with finite difference and finite element methods developed for coupled seismic and electromagnetic wave propagation in heterogeneous or poroviscoelastic media \cite{haines2006seismoelectric,santos2011finite,santos2012numerical};
related well-posedness results for electroseismic systems in anisotropic and inhomogeneous media can be found in \cite{mcghee2011class,mcghee2011electroseismic}.
For quasi-static electroporoelasticity, existing studies have established well-posedness and finite element error estimates for the four-field formulation \cite{hu2022numerical,hu2023existence}. 
We start from this formulation, which seeks the electric field \(\E\), magnetic field \(\HH\), displacement \(\uu\), and pore pressure \(p\) such that
\begin{subequations}
\label{eq:four}
\begin{align}
  \epsilon \partial_t \E + \sigma \E - \curl \HH - L\nabla p &= \hh,
    \label{eq:model-ampere}\\
  \mu \partial_t \HH + \curl \E &=0,
    \label{eq:model-faraday}\\
  -\lambda\nabla(\diver\uu)-G\Delta\uu+\alpha\nabla p &= \ff,
    \label{eq:model-momentum}\\
  \partial_t(c_0p+\alpha\diver\uu)-\kappa\Delta p+L\diver\E &= g.
    \label{eq:model-mass}
\end{align}
\end{subequations} 
Here, \(\epsilon\) is the electric permittivity, \(\sigma\) is the electric conductivity, \(\mu\) is the magnetic permeability, \(G\) is the shear modulus of the elastic skeleton, \(\lambda\) is the effective compressibility modulus, \(\alpha\) is the Biot--Willis coefficient, \(c_0\) is the storage coefficient, \(\kappa\) is the permeability coefficient, and \(L\) is the electrokinetic coupling coefficient between the electromagnetic and poroelastic equations.

Solving the coupled system \eqref{eq:four} monolithically can be computationally expensive, since the electromagnetic and poroelastic variables must be solved simultaneously at each time level. 
Splitting strategies are therefore introduced to reduce the computational cost. 
For example, a backward-Euler-based splitting scheme decouples the electromagnetic and poroelastic subproblems at each time level \cite{liu2025splitting}, while a physics-based multirate time-stepping strategy further reduces the overall computational expense by avoiding some costly poroelastic solves \cite{liu2026multitime}.
However, a single-pass splitting scheme generally produces a discrete solution that differs from that of the corresponding monolithic method and therefore introduces an additional splitting consistency error. 
Depending on the treatment of the coupling terms, additional stability restrictions may also arise, particularly for higher-order time discretizations \cite{cai2023some,altmann2024semi}. 
This motivates an iterative treatment similar to fixed-stress-type methods for poroelasticity, where the decoupled subproblems are repeatedly solved to recover consistency between the coupled variables \cite{kim2011stability,mikelic2013convergence,both2017robust,gu2023iterative}.
Such an approach retains the modularity of the splitting procedure while allowing the iterates to converge to the solution of the underlying monolithic discretization.

Another challenge lies in selecting finite element spaces that are compatible with the distinct mathematical structures of the electromagnetic and poroelastic subproblems.
In particular, the electromagnetic block calls for a curl-conforming approximation of the electric field \cite{monk2003finite}, whereas the poroelastic block requires a parameter-robust discretization to avoid volumetric locking in the nearly incompressible regime and to remain stable in the low-storage and low-permeability regimes \cite{lee2016robust}.
Moreover, pressure inaccuracies or spurious oscillations may propagate to the electromagnetic variables through the electrokinetic coupling and thereby affect the overall solution accuracy. 
To address these difficulties, we introduce the total pressure as an additional unknown and reformulate the mechanical equations as a generalized Stokes system \cite{oyarzua2016locking,lee2016robust}. This formulation permits stable and relatively standard finite element choices, such as a Taylor--Hood or MINI pair for the displacement and total pressure, together with a Lagrange space for the pore pressure.

The purpose of this paper is to develop and analyze stable and efficient finite element methods for the five-field reformulation. Starting from the weak formulation, we establish a continuous stability estimate and design a mixed finite element discretization using a curl-conforming N\'ed\'elec space for the electric field, a discontinuous vector polynomial space for the magnetic field, a Taylor--Hood pair for the displacement and total pressure, and a conforming Lagrange space for the pore pressure. 
Employing a BDF2 time discretization, we derive an a priori error estimate confirming second-order convergence in both space and time.
We then propose an iterative decoupling method that alternates between the electromagnetic and poroelastic subproblems and prove that it is contractive under an explicit coupling condition, thereby ensuring geometric convergence to the fully discrete monolithic solution.
Exploiting the fact that the discrete pressure gradient belongs to the N\'ed\'elec space and is curl-free, we further derive an algebraically equivalent reduced formulation. At each time level, the electromagnetic subsystem is solved once for initialization, after which only the poroelastic subsystem requires repeated solves. During the subsequent iterations, the magnetic field remains fixed, while the electric field is updated explicitly through a pressure-gradient correction.

The rest of the paper is organized as follows. 
Section~\ref{sec:2} presents the five-field reformulation and the weak formulation, together with a continuous stability estimate. 
Section~\ref{sec:3} develops fully discrete finite element approximations based on the BDF2 method, including the monolithic scheme, the iterative decoupling algorithm, and its reduced form; an a priori error estimate and the geometric convergence of the iteration are derived.
Section~\ref{sec:4} validates the proposed methods through convergence tests and iterative comparisons; a three-dimensional footing benchmark is also used to demonstrate the locking-free behavior of the total-pressure formulation.

\section{Five-field electroporoelasticity}
\label{sec:2}
Let \(\OmegaD\subset\mathbb{R}^3\) be a bounded Lipschitz polyhedron with boundary \(\partial\OmegaD\), and let \(T_f>0\) be the final time.
Introducing the total pressure $\xi := \alpha p-\lambda\diver\uu$ into the four-field system \eqref{eq:four} yields the equivalent five-field system in \(\OmegaD\times(0,T_f]\):
\begin{subequations}
\label{eq:five}
\begin{align}
  \epsilon \partial_t \E + \sigma \E - \curl \HH - L\nabla p &= \hh,
    \label{eq:five-e}\\
  \mu \partial_t \HH + \curl \E &=0,
    \label{eq:five-h}\\
  -G\Delta\uu+\nabla\xi &= \ff,
    \label{eq:five-u}\\
  \diver\uu+\lambda^{-1}\xi-\alpha\lambda^{-1}p &=0,
    \label{eq:five-xi}\\
  (c_0+\alpha^2\lambda^{-1})\partial_t p
  -\alpha\lambda^{-1}\partial_t\xi
  -\kappa\Delta p+L\diver\E &=g.
    \label{eq:five-p}
\end{align}
\end{subequations} 

To close the system, we shall prescribe boundary conditions and initial data. 
Let \(\partial\OmegaD=\overline{\Gamma_D^u}\cup\overline{\Gamma_N^u}\), where \(\Gamma_D^u\) and \(\Gamma_N^u\) are non-empty, disjoint boundary parts. We then impose
\begin{align*}
  \uu = \bm0 \quad \text{on } \Gamma_D^u,\qquad
  (G\nabla\uu-\xi \mathbf{I})\bn = \bm0 \quad \text{on } \Gamma_N^u,\qquad
  \E\times\bn = \bm0 \quad \text{on } \partial\OmegaD,\qquad
  p = 0 \quad \text{on } \partial\OmegaD .
\end{align*}
Here, \(\bn\) denotes the outward unit normal vector on \(\partial\OmegaD\).
No essential boundary condition is prescribed for the magnetic field or total pressure.
The system is supplemented with initial conditions at \(t=0\):
\begin{align*}
  \E(0) = \E^0,\quad
  \HH(0) = \HH^0,\quad
  \uu(0) = \uu^0,\quad
  p(0) = p^0,
\end{align*}
where the initial condition for \(\xi\) is determined by the compatibility relation
\(\xi^0 = \alpha p^0 - \lambda \diver \uu^0\).

\subsection{Weak formulation}
Let \(L^2(\OmegaD)\) be the space of square-integrable functions on \(\OmegaD\),
equipped with the standard inner product \(\ip{\cdot}{\cdot}\) and norm
\(\|\cdot\|_{L^2}\).
For an integer \(m\ge0\), \(H^m(\OmegaD)\) denotes
the standard Sobolev space with norm \(\|\cdot\|_{H^m}\).
In particular, \(H_0^1(\OmegaD)\) is the subspace of \(H^1(\OmegaD)\) with zero trace on
\(\partial\OmegaD\).
The bold notations
\(\bm L^2(\OmegaD)=[L^2(\OmegaD)]^3\) and
\(\bm H^m(\OmegaD)=[H^m(\OmegaD)]^3\) denote the corresponding vector-valued
spaces, equipped with the induced norms. 
For a Banach space \(X\), \(1\le s\le\infty\), and a time interval
\(I\subset[0,T_f]\), the Bochner space \(L^s(I;X)\) consists of measurable
functions \(v:I\to X\) with finite norm
\(\|v\|_{L^s(I;X)}:=(\int_I\|v(t)\|_X^s\,dt)^{1/s}\) for
\(1\le s<\infty\), while
\(\|v\|_{L^\infty(I;X)}:=\operatorname*{ess\,sup}_{t\in I}\|v(t)\|_X\).
We adopt the following notations:
\begin{align*}
  & \bm H(\text{curl};\OmegaD) =
  \{\bm v\in\bm L^2(\OmegaD): \nabla\times\bm v\in\bm L^2(\OmegaD)\}, 
  \\
  & \Ec=\{\bm v\in\bm H(\text{curl};\OmegaD): \bm v\times\bn=\bm0
  \hbox{ on }\partial\OmegaD\},\qquad
  \Hc=\bm L^2(\OmegaD), 
  \\
  & \Uc=\{\vv\in\bm H^1(\OmegaD):\vv=\bm0
  \hbox{ on }\Gamma_{D}^u\},\qquad
  \Wc=L^2(\OmegaD),\qquad
  \Pc=H_0^1(\OmegaD).
\end{align*}
The corresponding product space is
\(
  \mathbb S
  :=
  \Ec\times\Hc\times\Uc\times\Wc\times\Pc
\).

Throughout the rest of the paper, we have the following assumption.
\begin{assumption}
\label{ass:a-priori}
  The physical parameters are assumed to be positive: 
  \(
\epsilon,\sigma,\mu,G,\lambda,\alpha,\kappa, c_0 > 0.
  \)
  The coupling coefficient is bounded by 
  \(0<L<\sqrt{\sigma\kappa}\).
  The solution of \eqref{eq:five} is regular enough, and
  the source terms satisfy
  \(
    \ff\in H^1(0,T_f;\bm L^2(\OmegaD)), \
    \hh\in L^2(0,T_f;\bm L^2(\OmegaD)), \
    g\in L^2(0,T_f;L^2(\OmegaD)).
  \)
\end{assumption}

\noindent
The weak formulation of the five-field system \eqref{eq:five} is stated as follows: For almost every $t \in (0, T_f]$, find $(\E(t), \HH(t), \uu(t), \xi(t), p(t)) \in \mathbb S$ such that, for all \((\DD,\BB,\vv,w,q)\in\mathbb S\),
\begin{subequations}
\begin{align}
  \epsilon \ip{\partial_t\E}{\DD}
  +\sigma \ip{\E}{\DD}
  -\ip{\HH}{\curl\DD}
  -L\ip{\nabla p}{\DD}
  &=\ip{\hh}{\DD},
    \label{eq:weak-e}\\
  \mu \ip{\partial_t\HH}{\BB}
  +\ip{\curl\E}{\BB}
  &=0,
    \label{eq:weak-h}\\
  G\ip{\nabla\uu}{\nabla\vv}
  -\ip{\xi}{\diver\vv}
  &=\ip{\ff}{\vv},
    \label{eq:weak-u}\\
  \ip{\diver\uu}{w}
  +\lambda^{-1}\ip{\xi}{w}
  -\alpha\lambda^{-1}\ip{p}{w}
  &=0,
    \label{eq:weak-xi}\\
  (c_0+\alpha^2\lambda^{-1})\ip{\partial_t p}{q}
  -\alpha\lambda^{-1}\ip{\partial_t\xi}{q}
  +\kappa\ip{\nabla p}{\nabla q}
  -L\ip{\E}{\nabla q}
  &=\ip{g}{q}.
    \label{eq:weak-p}
\end{align}
\end{subequations} 

We note that the pair \((\Uc,\Wc)\) satisfies the continuous inf--sup condition for the divergence operator \cite{brenner1993nonconforming}, that is, there exists a constant \(\beta>0\), depending only on \(\OmegaD\) and the boundary decomposition, such that
\begin{equation}
  \beta \lnorm{w} 
  \le
  \sup_{\bm 0\ne\vv\in\Uc}
  \frac{\ip{\diver\vv}{w}}{\hnorm{\vv}}
  \qquad \forall\,w\in\Wc.
  \label{eq:continuous-infsup}
\end{equation}

\subsection{Continuous stability estimate}
\label{subsec:energy_estimate}

We now establish an energy-based stability estimate for the continuous five-field system.
For generic fields \((\DD,\BB,\vv,w,q) \in \mathbb S\), we define two functionals:
\begin{align*}
  \mathcal E (\DD,\BB,\vv,w,q)
  &:=
  \epsilon\lnorm{\DD}^2+\mu\lnorm{\BB}^2
  +G\lnorm{\nabla\vv}^2
  +\lambda^{-1}\lnorm{w-\alpha q}^2
  +c_0\lnorm{q}^2, \\
  \mathcal V(\DD, q) &:= \sigma \lnorm{\DD}^2 + \kappa \lnorm{\nabla q}^2.
\end{align*}
We use the notation \(A\lesssim B\) to denote \(A\le CB\), where \(C>0\) is independent of \(\lambda\). To derive an energy estimate without exponential dependence on \(T_f\), we use the following Gronwall-free estimate, motivated by \cite{lee2016robust,burger2021virtual}. A proof is included for completeness.
\begin{lemma}
\label{lem:continuous-non-gronwall}
Let \(Y\in C([0,T_f])\) be non-negative, and let \(F,D\) be non-negative functions on \([0,T_f]\), with \(D\) non-decreasing.  
Suppose that there exists a constant \(C_0\ge1\) such that
\begin{equation*}
  Y(t)^2
  \le
  C_0Y(0)^2+D(t)
  +\int_0^tF(s)Y(s) ds,
\end{equation*}
for every \(t\in[0,T_f]\). Then, the following estimate holds
\begin{equation*}
  Y(t)^2
  \lesssim
  Y(0)^2+D(t)+\int_0^tF(s)^2 ds.
\end{equation*}
\end{lemma}
\begin{proof}
Fix \(t\in[0,T_f]\) and set \(M_t:=\max_{0\le s\le t}Y(s)\). Since \(D\) is non-decreasing, taking the maximum over \(s\in[0,t]\) in the initial assumption yields
\[
  M_t^2 \le C_0Y(0)^2+D(t) + M_t\int_0^t F(r) dr.
\]
Applying Young's inequality to the integral term, we obtain
\[
  M_t^2
  \le 2C_0Y(0)^2+2D(t)
  +\left(\int_0^tF(r) dr\right)^2 .
\]
To bound the remaining terms on the right-hand side, we apply the Cauchy--Schwarz inequality to the integral:
\[
  \left(\int_0^tF(r) dr\right)^2
  \le t \int_0^tF(r)^2 dr
  \le T_f\int_0^tF(r)^2 dr.
\]
Combining these estimates with the obvious bound $Y(t) \le M_t$, we derive the desired estimate.
\end{proof}

\begin{theorem}
\label{thm:continuous-energy}
Let \((\E,\HH,\uu,\xi,p)\) be a sufficiently regular solution of \eqref{eq:weak-e}--\eqref{eq:weak-p} on \([0,T_f]\). Then, for \(t\in[0,T_f]\),
\begin{align}
  &\mathcal E(\E(t),\HH(t),\uu(t),\xi(t),p(t))
  +\left(1-\frac{L}{\sqrt{\sigma\kappa}}\right)\int_0^{t}
  \mathcal V(\E(s), p(s))ds
  \nonumber\\
  &\qquad\lesssim
    \mathcal  E(\E^0,\HH^0,\uu^0,\xi^0,p^0)
    +\|\ff\|_{H^1(0,T_f;\bm L^2)}^2
    +\|\hh\|_{L^2(0,T_f;\bm L^2)}^2
    +\|g\|_{L^2(0,T_f;L^2)}^2.
  \label{eq:energy_estimate_bound}
\end{align}
Moreover,
\begin{equation}
  \lnorm{\xi(t)}
  \lesssim G\lnorm{\nabla\uu(t)}+\lnorm{\ff(t)},
  \label{eq:xi_bound_infsup}
\end{equation}
\end{theorem}

\begin{proof}
Differentiating the algebraic constraint \eqref{eq:weak-xi} with respect to time
gives
\begin{equation}
  \ip{\diver\partial_t\uu}{w}
  +\lambda^{-1}\ip{\partial_t\xi}{w}
  -\alpha\lambda^{-1}\ip{\partial_t p}{w}
  =0
  \qquad \forall w\in\Wc .
  \label{eq:weak-xi2}
\end{equation}
We also note the following algebraic identity
\begin{align}
  &\alpha^2 \ip{\partial_t p}{p}
  -\alpha\ip{\partial_t\xi}{p}
  +\ip{\partial_t\xi}{\xi}
  -\alpha\ip{\partial_t p}{\xi}
  =
  \ip{\partial_t\xi-\alpha\partial_t p}{\xi-\alpha p}.
  \label{eq:pressure_total_pressure_identity}
\end{align}

Taking
\(\DD=\E\) in \eqref{eq:weak-e},
\(\BB=\HH\) in \eqref{eq:weak-h},
\(\vv=\partial_t\uu\) in \eqref{eq:weak-u},
\(w=\xi\) in \eqref{eq:weak-xi2}, and
\(q=p\) in \eqref{eq:weak-p}, and then summing the resulting identities with \eqref{eq:pressure_total_pressure_identity}, we
obtain
\begin{align}
  &\epsilon\ip{\partial_t\E}{\E}
  +\sigma\lnorm{\E}^2
  +\mu\ip{\partial_t\HH}{\HH}
  +G\ip{\nabla\uu}{\nabla\partial_t\uu}
  +\lambda^{-1}\ip{\partial_t\xi-\alpha\partial_t p}{\xi-\alpha p}
  \nonumber\\
  &\quad
  +c_0\ip{\partial_t p}{p}
  +\kappa\lnorm{\nabla p}^2
  -2L\ip{\E}{\nabla p}
  =
  \ip{\hh}{\E}+\ip{\ff}{\partial_t\uu}+\ip{g}{p}.
  \label{eq:energy_expanded_sum}
\end{align}
Using \((\partial_t \chi,\chi)=\frac12\frac{d}{dt}\lnorm{\chi}^2\) to \eqref{eq:energy_expanded_sum}, we derive the
differential identity
\begin{equation}
  \frac12\frac{d}{dt}\mathcal E(\E,\HH,\uu,\xi,p)
  + \mathcal V(\E, p)
  -2L\ip{\E}{\nabla p}
  =
  \ip{\hh}{\E}+\ip{\ff}{\partial_t\uu}+\ip{g}{p}.
  \label{eq:energy_differential_identity}
\end{equation}
Integrating \eqref{eq:energy_differential_identity} over \((0,t)\) gives
\begin{align}
  &\frac12\mathcal E(\E(t),\HH(t),\uu(t),\xi(t),p(t))
  +\int_0^t
  \left[
  \mathcal V(\E, p)
  -2L\ip{\E}{\nabla p}
  \right] ds
  \nonumber\\
  &\qquad =
  \frac12\mathcal E(\E^0,\HH^0,\uu^0,\xi^0,p^0)
  +\int_0^t
   \left[
    \ip{\hh}{\E}+\ip{\ff}{\partial_t\uu}+\ip{g}{p}
   \right] ds .
  \label{eq:energy_integral_identity}
\end{align}

For the left-hand side of \eqref{eq:energy_integral_identity}, we apply Cauchy--Schwarz inequality to derive
\begin{equation*}
  2L \ip{\E}{\nabla p} 
  \le
  \frac{L}{\sqrt{\sigma\kappa}}
  \left(\sigma\lnorm{\E}^2+\kappa \lnorm{\nabla p}^2\right) = \frac{L}{\sqrt{\sigma\kappa}} \mathcal V(\E, p),
\end{equation*}
leading to
\begin{equation}
  \left(1- \frac{L}{\sqrt{\sigma\kappa}}\right) \mathcal V(\E, p) \leq \mathcal V(\E, p) - 2L \ip{\E}{\nabla p}.
  \label{eq:coupled_dissipation_bound}
\end{equation}
We next estimate the right-hand side of
\eqref{eq:energy_integral_identity}. By Cauchy--Schwarz, Young, and Poincar{\'e}
inequalities,
\begin{align}
  \int_0^t\ip{\hh}{\E} ds
  &\le
  C\int_0^t\lnorm{\hh}^2 ds
  +\frac14\left(1-\frac{L}{\sqrt{\sigma\kappa}}\right)
   \sigma\int_0^t\lnorm{\E}^2 ds, \label{eq:h_source_estimates} \\
  \int_0^t\ip{g}{p} ds
  &\le
  C\int_0^t\lnorm{g}^2 ds
  +\frac14\left(1-\frac{L}{\sqrt{\sigma\kappa}}\right)
   \kappa\int_0^t\lnorm{\nabla p}^2 ds .
  \label{eq:g_source_estimates}
\end{align}
For the forcing term in the momentum equation, integration by parts in time
yields
\begin{align}
  & \int_0^t\ip{\ff}{\partial_t\uu} ds
  =
  \ip{\ff(t)}{\uu(t)}
  -\ip{\ff(0)}{\uu(0)}
  -\int_0^t\ip{\partial_t\ff}{\uu} ds
  \nonumber\\
  & \le
  \frac{G}{4}\lnorm{\nabla\uu(t)}^2
  +\frac{G}{4}\lnorm{\nabla\uu^0}^2
  +C\|\ff\|_{L^\infty(0,T_f;\bm L^2)}^2
  +C\int_0^t\lnorm{\partial_t\ff(s)}\,\lnorm{\nabla\uu(s)} ds .
  \label{eq:f_source_estimate}
\end{align}
Substituting \eqref{eq:coupled_dissipation_bound}--\eqref{eq:f_source_estimate} into \eqref{eq:energy_integral_identity}, absorbing the endpoint term \(G\lnorm{\nabla\uu(t)}^2/4\) and the terms involving  \(\mathcal V(\E, p)\) into the left-hand side, and incorporating \(G\|\nabla\uu^0\|^2/4\) into the initial energy, we obtain
\begin{align}
  Y(t)^2
  \le
  Y(0)^2
  +C\|\ff\|_{L^\infty(0,T_f;\bm L^2)}^2
  +C\int_0^t\left(\lnorm{\hh}^2+\lnorm{g}^2\right)ds
  +C\int_0^t\lnorm{\partial_t\ff(s)}Y(s) ds ,
  \label{eq:before-non-gronwall}
\end{align}
where
\[
  Y(t)^2 :=
  \mathcal E(\E(t),\HH(t),\uu(t),\xi(t),p(t))
  +\left(1-\frac{L}{\sqrt{\sigma\kappa}}\right)
  \int_0^t\mathcal V(\E(s),p(s))\,ds .
\]
Applying Lemma~\ref{lem:continuous-non-gronwall} to \eqref{eq:before-non-gronwall}, and using
\(\|\ff\|_{L^\infty(0,T_f;\bm L^2)}
\lesssim \|\ff\|_{H^1(0,T_f;\bm L^2)}\), gives
\eqref{eq:energy_estimate_bound}.

It remains to bound \(\xi\). By the inf-sup condition
\eqref{eq:continuous-infsup} and the momentum equation \eqref{eq:weak-u},
\[
  \beta\lnorm{\xi}
  \le
  \sup_{\bm0\ne\vv\in\Uc}
  \frac{\ip{\xi}{\diver\vv}}{\hnorm{\vv}}
  =
  \sup_{\bm0\ne\vv\in\Uc}
  \frac{G\ip{\nabla\uu}{\nabla\vv}-\ip{\ff}{\vv}}{\hnorm{\vv}}
  \le
  C\left(G\lnorm{\nabla\uu}+\lnorm{\ff}\right),
\]
where the Poincar{\'e} inequality is applied.
This proves \eqref{eq:xi_bound_infsup}.
\end{proof}

\section{Fully discrete schemes}
\label{sec:3}
\subsection{Spatial and temporal discretization}
\label{STD}

Let \(\mathcal T_h\) be a shape-regular tetrahedral partition of \(\OmegaD\), consisting of disjoint elements \(K\). The mesh size is denoted by $h:=\max_{K\in\mathcal T_h}\operatorname{diam}(K). $
For \(j\ge0\), let \(P_j(K)\) be the space of scalar polynomials of total degree at most \(j\) on \(K\), and let \(\widetilde P_j(K)\) be the corresponding space of homogeneous polynomials of degree \(j\). Their vector-valued counterparts are denoted by \(\bm P_j(K)\) and \(\widetilde{\bm P}_j(K)\), respectively. Following
\cite{nedelec1980mixed,monk2003finite,liu2025splitting}, 
we define the local N\'ed\'elec space as follows.
\[
  \bm S_2(K)=\{\bm r\in\widetilde{\bm P}_2(K):
  \bm r(\bm x)\cdot \bm x=0\},\qquad
  \bm R_2(K)=\bm P_1(K)\oplus \bm S_2(K).
\]
The finite element spaces are then defined by
\begin{align*}
  \Eh
  &=\{\DD_h\in\Ec:
      \DD_h|_K\in \bm R_2(K),\ K\in\mathcal T_h\},\\
  \Hh
  &=\{\BB_h\in\Hc:
      \BB_h|_K\in\bm P_1(K),\ K\in\mathcal T_h\},\\
  \Uh
  &=\{\vv_h\in\Uc\cap \bm C^0(\overline{\OmegaD}):
      \vv_h|_K\in\bm P_2(K),\ K\in\mathcal T_h\},\\
  \Wh
  &=\{w_h\in\Wc\cap C^0(\overline{\OmegaD}):
      w_h|_K\in P_1(K),\ K\in\mathcal T_h\},\\
  \Ph
  &=\{q_h\in\Pc\cap C^0(\overline{\OmegaD}):
      q_h|_K\in P_2(K),\ K\in\mathcal T_h\}.
\end{align*}
The corresponding product space is 
\(
  \mathbb S_h=\Eh\times\Hh\times\Uh\times\Wh\times\Ph
\).
The electromagnetic spaces satisfy the compatibility
\(\curl\Eh\subset\Hh\). Indeed, if \(\DD_h\in\Eh\), then
\(\DD_h|_K\in\bm R_2(K)=\bm P_1(K)\oplus\bm S_2(K)\) on each element
\(K\). The curl of the \(\bm P_1(K)\) part belongs to \(\bm P_0(K)\),
whereas the curl of the homogeneous degree-two part in \(\bm S_2(K)\)
belongs to \(\bm P_1(K)\). Hence
\[
  (\curl\DD_h)|_K\in\bm P_1(K)
  \qquad \forall K\in\mathcal T_h .
\]
Since \(\Hh\) is the discontinuous \(\bm P_1\) vector space, no interelement
continuity is required, and therefore \(\curl\Eh\subset\Hh\).

The Taylor--Hood pair \((\Uh,\Wh)\) is uniformly inf-sup stable: there
exists a constant \(\tilde \beta\), independent of \(h\), such that
\begin{equation}
  \tilde \beta \lnorm{ w_h }
  \le
  \sup_{\bm0\ne\vv_h\in\Uh}
  \frac{\ip{\diver\vv_h}{w_h}}{\hnorm{ \vv_h }}
  \qquad \forall\,w_h\in\Wh .
  \label{disc:infsup}
\end{equation}

\paragraph{Projection estimates.} 
For the electric field \(\E \in \Ec \), let \(\Pi_h^E\E\in\Eh\) denote its N\'ed\'elec interpolation. If, in addition, $\E \in \bm H^3(\OmegaD)$, the following estimate holds
\begin{equation}
  \lnorm{\E-\Pi_h^E\E}
  +\lnorm{\curl(\E-\Pi_h^E\E)}
  \le
  Ch^2\|\E\|_{H^3} .
  \label{eq:hcurl-projection-estimate}
\end{equation}
For the magnetic field \(\HH \in \Hc \), use the \(L^2\)-projection
\(\Pi_h^H\HH\in\Hh\), satisfying
\begin{equation}
  \ip{\HH - \Pi_h^H\HH}{\BB_h}=0
  \qquad
  \forall\,\BB_h\in\Hh .
  \label{eq:l2-h-projection}
\end{equation}
Since \(\curl\DD_h\in\Hh\) for every \(\DD_h\in\Eh\), the same
\(L^2\)-projection gives the electromagnetic orthogonality
$\ip{\HH-\Pi_h^H\HH}{\curl\DD_h}=0$.
If, in addition, \(\HH\in \bm H^2(\OmegaD)\), then
\begin{equation}
  \lnorm{\HH-\Pi_h^H\HH}
  \le Ch^2\|\HH\|_{H^2} .
  \label{eq:l2-h-projection-estimate}
\end{equation}
For variables \((\uu,\xi)\), the Stokes-like projection \((\Pi_h^U\uu,\Pi_h^W\xi)\in\Uh\times\Wh\) is defined by
\begin{align}
  G\ip{\nabla(\uu - \Pi_h^U\uu)}{\nabla\vv_h}
  -\ip{\xi-\Pi_h^W\xi}{\diver\vv_h}
  &=0
  \qquad \forall\,\vv_h\in\Uh,
  \label{eq:stokes-projection-u}\\
  \ip{\diver(\uu - \Pi_h^U\uu)}{w_h}
  &=0
  \qquad \forall\,w_h\in\Wh .
  \label{eq:stokes-projection-xi}
\end{align}
If, in addition, \((\uu,\xi)\in \bm H^3(\OmegaD)\times H^2(\OmegaD)\), then the
approximation estimate holds
\begin{equation}
  \hnorm{\uu-\Pi_h^U\uu}
  +\lnorm{\xi-\Pi_h^W\xi}
  \le
  Ch^2
  (
  \|\uu\|_{H^3}
  +\|\xi\|_{H^2}
  ).
  \label{eq:stokes-projection-estimate}
\end{equation}
For the pore pressure \(p \in \Pc \), use the elliptic projection
\(\Pi_h^P p\in\Ph\), defined by
\begin{equation}
  \kappa\ip{\nabla(p - \Pi_h^P p)}{\nabla q_h}=0
  \qquad
  \forall\, q_h\in\Ph .
  \label{eq:pressure-elliptic-projection}
\end{equation}
If, in addition, \(p\in H^3(\OmegaD)\), then
\begin{equation}
  \hnorm{p-\Pi_h^P p}
  \le Ch^2\|p\|_{H^3} .
  \label{eq:pressure-elliptic-estimate}
\end{equation}
If \(\OmegaD\) has full elliptic regularity, then the Aubin--Nitsche argument yields
\begin{equation}
  \lnorm{p-\Pi_h^P p}
  \le Ch^3\|p\|_{H^3} .
  \label{eq:pressure-l2-estimate}
\end{equation}

We approximate the time derivatives by the second-order backward differentiation formula (BDF2), a two-step implicit finite difference method.
Let \(0=t_0<t_1<\cdots<t_N=T_f\) be a uniform partition of \([0,T_f]\), with time-step size \(\dt=T_f/N\). For any time-dependent quantity \(\chi\), we denote \(\chi^n:=\chi(t_n)\). For \(n\ge2\), the BDF2 difference operator is defined by
\begin{equation*}
  \bar{\partial}_t \chi^n
  =\frac{3\chi^n-4\chi^{n-1}+\chi^{n-2}}{2\dt}.
\end{equation*}
The stability of BDF2 is expressed through the following
\(\mathcal G\)-norm. Introduce the symmetric positive definite matrix
\[
  \mathcal G=
  \begin{bmatrix}
    \frac12&-1\\
    -1&\frac52
  \end{bmatrix},
\]
and denote its eigenvalues by
\(\lambda_{\min}<\lambda_{\max}\). For a Hilbert-space-valued sequence
\(\{\chi^n\}_{n\ge0}\), set
\[
  [\chi^{n};\chi^{n-1}]:=(\chi^{n},\chi^{n-1})^T,
  \qquad
  \|[\chi^{n};\chi^{n-1}]\|_{\mathcal G}^2
  :=
  \bigl([\chi^{n};\chi^{n-1}],
  \mathcal G[\chi^{n};\chi^{n-1}]\bigr).
\]
Then, we have $\mathcal G$-norm equivalence
\begin{align}
  \lambda_{\min}
  \left(\lnorm{\chi^{n}}^2+\lnorm{\chi^{n-1}}^2\right)
  &\le
  \|[\chi^{n};\chi^{n-1}]\|_{\mathcal G}^2
  \le
  \lambda_{\max}
  \left(\lnorm{\chi^{n}}^2+\lnorm{\chi^{n-1}}^2\right),
  \label{eq:G-norm-equivalence}
\end{align}
and
\begin{align}
  \ip{\bar{\partial}_t\chi^n}{\chi^n}
  ={}&
  \frac{1}{2\dt}
  \left(
    \|[\chi^{n-1};\chi^n]\|_{\mathcal G}^2
    -
    \|[\chi^{n-2};\chi^{n-1}]\|_{\mathcal G}^2
  \right) +
  \frac{1}{4\dt}
  \lnorm{\chi^n-2\chi^{n-1}+\chi^{n-2}}^2.
  \label{eq:bdf2-G-identity}
\end{align}
Since BDF2 is a two-step method, the approximations at \(t_0\) and \(t_1\) are required before the recurrence starts. The initial discrete data at \(t_0\) are chosen to satisfy the algebraic constraint
\[
(\nabla\cdot \bm u_h^0,w_h)
+\lambda^{-1}(\xi_h^0,w_h)
-\alpha\lambda^{-1}(p_h^0,w_h)=0
\qquad \forall w_h\in W_h.
\]
The value at \(t_1\) is computed by one monolithic backward Euler step, which contains the same algebraic constraint and therefore preserves the corresponding discrete compatibility condition. For the subsequent error analysis, we assume that the starting values at \(t_0\) and \(t_1\) satisfy the discrete compatibility condition and are second-order accurate in the component norms entering the discrete energy introduced below.

\subsection{Monolithic approximation}
We first describe the fully coupled scheme.
For \(n\ge2\), given the previous time-level solutions, find
$(\E_h^n,\HH_h^n,\uu_h^n,\xi_h^n,p_h^n)\in\mathbb S_h$
such that, for all
\((\DD_h,\BB_h,\vv_h,w_h,q_h)\in\mathbb S_h\),
\begin{subequations} 
\label{eq:mono}
\begin{align}
  \epsilon \ip{\bar{\partial}_t\E_h^n}{\DD_h}
  + \sigma \ip{\E_h^n}{\DD_h}
  -\ip{\HH_h^n}{\curl\DD_h}
  -L\ip{\nabla p_h^n}{\DD_h}
  &=\ip{\hh^n}{\DD_h},
    \label{eq:mono-e}\\
  \mu \ip{\bar{\partial}_t\HH_h^n}{\BB_h}
  +\ip{\curl\E_h^n}{\BB_h}
  &=0,
    \label{eq:mono-h}\\
  G\ip{\nabla\uu_h^n}{\nabla\vv_h}
  -\ip{\xi_h^n}{\diver\vv_h}
  &=\ip{\ff^n}{\vv_h},
    \label{eq:mono-u}\\
  \ip{\diver\uu_h^n}{w_h}
  +\lambda^{-1}\ip{\xi_h^n}{w_h}
  -\alpha\lambda^{-1}\ip{p_h^n}{w_h}
  &=0,
    \label{eq:mono-xi}\\
  (c_0+\alpha^2\lambda^{-1})
    \ip{\bar{\partial}_t p_h^n}{q_h}
  -\alpha\lambda^{-1}\ip{\bar{\partial}_t\xi_h^n}{q_h}
  +\kappa\ip{\nabla p_h^n}{\nabla q_h}
  -L\ip{\E_h^n}{\nabla q_h}
  &=\ip{g^n}{q_h}.
    \label{eq:mono-p}
\end{align}
\end{subequations} 
With the projection operators and initialization specified above, we then proceed to a priori error analysis. For simplicity, we define the discrete energy and dissipation functionals by
\begin{align*}
  \mathcal E^{\mathcal G}_h(\DD_h^n,\BB_h^n,\vv_h^n,w_h^n,q_h^n)
  :={}&
  \epsilon\|[\DD_h^{n};\DD_h^{n-1}]\|_{\mathcal G}^2
  +\mu\|[\BB_h^{n};\BB_h^{n-1}]\|_{\mathcal G}^2
  + G\|[\nabla\vv_h^{n};\nabla\vv_h^{n-1}]\|_{\mathcal G}^2
  \notag\\
  &+ c_0\|[q_h^{n};q_h^{n-1}]\|_{\mathcal G}^2 +
  \lambda^{-1}
  \|[w_h^{n}-\alpha q_h^{n};
      w_h^{n-1}-\alpha q_h^{n-1}]\|_{\mathcal G}^2,
  \\
  \mathcal V_h(\DD_h^n,q_h^n)
  :={}&
  \sigma\lnorm{\DD_h^n}^2
  +\kappa\lnorm{\nabla q_h^n}^2.
\end{align*}
To establish the error estimates for the fully discrete scheme, we state the following discrete analogue of Lemma \ref{lem:continuous-non-gronwall}. We omit the proof, as it follows by applying the exact same arguments as in the continuous case, simply replacing time integrals with discrete sums.
\begin{lemma}
\label{lem:discrete-non-gronwall}
Let \(\{Y_n\}_{n=1}^N\), \(\{F_n\}_{n=2}^N\), and
\(\{D_n\}_{n=1}^N\) be non-negative sequences, with \(\{D_n\}_{n=1}^N\)
non-decreasing. Suppose that there exists a constant \(C_0\ge1\) such that
\begin{equation*}
  Y_n^2
  \le
  C_0Y_1^2+D_n+\dt\sum_{j=2}^nF_jY_j,
\end{equation*}
for every \(2\le n\le N\). Then, the following estimate holds
\begin{equation*}
  Y_n^2
  \lesssim
  Y_1^2+D_n+\dt\sum_{j=2}^nF_j^2.
\end{equation*}
\end{lemma}

\begin{theorem}
\label{thm:monolithic-bdf2-error}
Let $(\E(t),\HH(t),\uu(t),\xi(t),p(t)) \in \mathbb S$
be the exact five-field solution of
\eqref{eq:five-e}--\eqref{eq:five-p}, and let
$(\E_h^n,\HH_h^n,\uu_h^n,\xi_h^n,p_h^n) \in \mathbb S_h$
be the monolithic BDF2 solution of
\eqref{eq:mono-e}--\eqref{eq:mono-p}. 
Then, for $2\leq n \leq N$, 
\begin{align}
  &\lnorm{\E(t_n)-\E_h^n}
  +\lnorm{\HH(t_n)-\HH_h^n}
  +\hnorm{\uu(t_n)-\uu_h^n}
  +\lnorm{\xi(t_n)-\xi_h^n}
  +\lnorm{p(t_n)-p_h^n}
  \nonumber\\
  &\quad
  +\bigg(\dt\sum_{j=2}^n
  \hnorm{p(t_j)-p_h^j}^2\bigg)^{1/2}
  \lesssim h^2+\dt^2 .
  \label{eq:monolithic-bdf2-error}
\end{align}
\end{theorem}

\begin{proof}
For \(\chi\in\{\E,\HH,\uu,\xi,p\}\), we split the error into the projection error and the discrete error:
\[
  e_\chi^n
  :=
  \chi(t_n)-\chi_h^n
  =
  \bigl(\chi(t_n)-\Pi_h^\chi\chi(t_n)\bigr)
  +
  \bigl(\Pi_h^\chi\chi(t_n)-\chi_h^n\bigr)
  =:
  e_\chi^{I,n}+e_\chi^{h,n},
\]
where boldface notation is used for the vector-valued variables. 
We also introduce the residual
\[
  \rho_\chi^n
  :=
  \partial_t\chi(t_n)
  -\bar{\partial}_t\Pi_h^\chi\chi(t_n)
  =
  \bigl(\partial_t\chi(t_n)-\bar{\partial}_t\chi(t_n)\bigr)
  +\bar{\partial}_t e_\chi^{I,n}.
\]
Taylor's formula and the time-independent projection operators give
\begin{align}
  \lnorm{\rho_\chi^n}^2
  \lesssim
  \dt^3
  \int_{t_{n-2}}^{t_n}
  \lnorm{\partial_t^3\chi(s)}^2 ds +
  \frac{1}{\dt}
  \int_{t_{n-2}}^{t_n}
  \lnorm{\partial_t e_\chi^{I}(s)}^2 ds .
  \label{eq:appendix-rho-estimate}
\end{align}

Subtracting the fully discrete approximation obtained by
\eqref{eq:mono-e}--\eqref{eq:mono-p} from the continuous weak formulation
\eqref{eq:weak-e}--\eqref{eq:weak-p} evaluated at \(t=t_n\), and using the
projection identities \eqref{eq:l2-h-projection}, \eqref{eq:stokes-projection-u}, \eqref{eq:stokes-projection-xi} and \eqref{eq:pressure-elliptic-projection}, yields the following system for the discrete errors:
\begin{subequations}
\label{eq:appendix-discrete-error-system}
\begin{align}
  \epsilon\ip{\bar{\partial}_t\bm e_E^{h,n}}{\DD_h}
  +\sigma\ip{\bm e_E^{h,n}}{\DD_h}
  -\ip{\bm e_H^{h,n}}{\curl\DD_h}
  -L\ip{\nabla e_p^{h,n}}{\DD_h}
  & = \mathcal R_E^n(\DD_h),
  \label{eq:appendix-discrete-e}
  \\
  \mu\ip{\bar{\partial}_t\bm e_H^{h,n}}{\BB_h}
  +\ip{\curl\bm e_E^{h,n}}{\BB_h}
  & = \mathcal R_H^n(\BB_h),
  \label{eq:appendix-discrete-h}
  \\
  G\ip{\nabla\bm e_u^{h,n}}{\nabla\vv_h}
  -\ip{e_\xi^{h,n}}{\diver\vv_h}
  & =0,
  \label{eq:appendix-discrete-u}
  \\
  \ip{\diver\bm e_u^{h,n}}{w_h}
  +\lambda^{-1}\ip{e_\xi^{h,n}}{w_h}
  -\alpha\lambda^{-1}\ip{e_p^{h,n}}{w_h}
  & = \mathcal R_\xi^n(w_h),
  \label{eq:appendix-discrete-xi}
  \\
  (c_0+\alpha^2\lambda^{-1})
  \ip{\bar{\partial}_t e_p^{h,n}}{q_h}
  -\alpha\lambda^{-1}
  \ip{\bar{\partial}_t e_\xi^{h,n}}{q_h}
  +\kappa\ip{\nabla e_p^{h,n}}{\nabla q_h}
  -L\ip{\bm e_E^{h,n}}{\nabla q_h}
  & = \mathcal R_p^n(q_h),
  \label{eq:appendix-discrete-p}
\end{align}
\end{subequations}
where the right-hand terms are:
\begin{subequations}
\begin{align*}
\mathcal R_E^n(\DD_h) & =  -\epsilon\ip{\bm\rho_E^n}{\DD_h}
  -\sigma\ip{\bm e_E^{I,n}}{\DD_h}
  +L\ip{\nabla e_p^{I,n}}{\DD_h},
\\ 
\mathcal R_H^n(\BB_h) & = -\mu\ip{\bm\rho_H^n}{\BB_h}
  -\ip{\curl\bm e_E^{I,n}}{\BB_h},
\\
\mathcal R_\xi^n(w_h) & = -\lambda^{-1}\ip{e_\xi^{I,n}}{w_h}
  +\alpha\lambda^{-1}\ip{e_p^{I,n}}{w_h},
\\
\mathcal R_p^n(q_h) & = -(c_0+\alpha^2\lambda^{-1})\ip{\rho_p^n}{q_h}
  +\alpha\lambda^{-1}\ip{\rho_\xi^n}{q_h}
  +L\ip{\bm e_E^{I,n}}{\nabla q_h}.
\end{align*}
\end{subequations}
We then apply the BDF2 difference operator to
\eqref{eq:appendix-discrete-xi} and obtain
\begin{align}
  \ip{\diver\bar{\partial}_t\bm e_u^{h,n}}{w_h}
  +\lambda^{-1}
  \ip{\bar{\partial}_t e_\xi^{h,n}}{w_h}
  -\alpha\lambda^{-1}
  \ip{\bar{\partial}_t e_p^{h,n}}{w_h}
  &=
  \bar{\partial}_t\mathcal R_\xi^n(w_h).
  \label{eq:appendix-differentiated-xi}
\end{align}
We choose
\(
  (\DD_h,\BB_h,\vv_h,w_h,q_h)
  =
  (\bm e_E^{h,n},
    \bm e_H^{h,n},
    \bar{\partial}_t\bm e_u^{h,n},
    e_\xi^{h,n},
    e_p^{h,n})
\)
in \eqref{eq:appendix-discrete-e},
\eqref{eq:appendix-discrete-h},
\eqref{eq:appendix-discrete-u},
\eqref{eq:appendix-differentiated-xi}, and
\eqref{eq:appendix-discrete-p}, respectively. Summing the resulting
identities, the Maxwell curl terms and the divergence terms cancel, and we obtain
\begin{align}
&\epsilon\ip{\bar{\partial}_t\bm e_E^{h,n}}{\bm e_E^{h,n}}
+\mu\ip{\bar{\partial}_t\bm e_H^{h,n}}{\bm e_H^{h,n}}
+G\ip{\nabla\bar{\partial}_t\bm e_u^{h,n}}
       {\nabla\bm e_u^{h,n}}
+c_0\ip{\bar{\partial}_t e_p^{h,n}}{e_p^{h,n}}
\notag\\
&\quad
+\lambda^{-1}
\ip{\bar{\partial}_t(e_\xi^{h,n}-\alpha e_p^{h,n})}
   {e_\xi^{h,n}-\alpha e_p^{h,n}}
+\sigma\lnorm{\bm e_E^{h,n}}^2
+\kappa\lnorm{\nabla e_p^{h,n}}^2
-2L\ip{\bm e_E^{h,n}}{\nabla e_p^{h,n}}
\notag\\
&=
\mathcal R_E^n(\bm e_E^{h,n})
+\mathcal R_H^n(\bm e_H^{h,n})
+\bar{\partial}_t\mathcal R_\xi^n(e_\xi^{h,n})
+\mathcal R_p^n(e_p^{h,n}).
\label{eq:appendix-error-energy}
\end{align}
We then apply Cauchy--Schwarz inequality to derive
\begin{align}
\left(1-\frac{L}{\sqrt{\sigma\kappa}} \right)\mathcal V_h(\bm e_E^{h,n},e_p^{h,n})
\leq
\mathcal V_h(\bm e_E^{h,n},e_p^{h,n})
-2L\ip{\bm e_E^{h,n}}{\nabla e_p^{h,n}}.
\label{eq:appendix-coupled-coercivity}
\end{align}
For simplicity, set
\[
  e_h^n
  :=
  (\bm e_E^{h,n},\bm e_H^{h,n},\bm e_u^{h,n},
   e_\xi^{h,n},e_p^{h,n}).
\]
Applying \eqref{eq:bdf2-G-identity} and \eqref{eq:appendix-coupled-coercivity} to \eqref{eq:appendix-error-energy}, omitting the nonnegative
second-difference terms, multiplying by \(2\dt\), and summing over \(n=2,\ldots,m\), we obtain, for \(2\le m\le N\),
\begin{align}
&\mathcal E^{\mathcal G}_h(e_h^m)
+2\dt\left(1-\frac{L}{\sqrt{\sigma\kappa}}\right)
\sum_{n=2}^m
  \mathcal V_h(\bm e_E^{h,n},e_p^{h,n})
\notag\\
&\le
\mathcal E^{\mathcal G}_h(e_h^1)
+2\dt\sum_{n=2}^m
\left(
  \mathcal R_E^n(\bm e_E^{h,n})
  +\mathcal R_H^n(\bm e_H^{h,n})
  +\bar{\partial}_t\mathcal R_\xi^n(e_\xi^{h,n})
  +\mathcal R_p^n(e_p^{h,n})
\right).
\label{eq:appendix-summed-error-energy}
\end{align}
To apply Lemma \ref{lem:discrete-non-gronwall}, we define 
\begin{align*}
  Y_m^2
  :={}&
  \mathcal E^{\mathcal G}_h(e_h^m)
  +\dt\left(1-\frac{L}{\sqrt{\sigma\kappa}}\right)
  \sum_{n=2}^m
  \mathcal V_h(\bm e_E^{h,n},e_p^{h,n}), \qquad Y_1^2 := \mathcal E^{\mathcal G}_h(e_h^1),
\\
D_m
:={}&
C\dt\sum_{n=2}^m
\bigg(
\frac{
  (
    \epsilon\lnorm{\bm\rho_E^n}
    +\sigma\lnorm{\bm e_E^{I,n}}
    +L\lnorm{\nabla e_p^{I,n}}
  )^2
}{
  (1-L/\sqrt{\sigma\kappa})\sigma
}
\notag\\
&\qquad\qquad
+
\frac{
  (
    (c_0+\alpha^2\lambda^{-1})\lnorm{\rho_p^n}
    +\alpha\lambda^{-1}\lnorm{\rho_\xi^n}
    +L\lnorm{\bm e_E^{I,n}}
  )^2
}{
  (1-L/\sqrt{\sigma\kappa})\kappa
}
\bigg),\\
F_n
:={}&
C\bigg(
  \mu\lnorm{\bm\rho_H^n}
  +\lnorm{\curl\bm e_E^{I,n}}
  +\lambda^{-1}\lnorm{\bar{\partial}_t e_\xi^{I,n}}
  +\alpha\lambda^{-1}\lnorm{\bar{\partial}_t e_p^{I,n}}
\bigg).
\end{align*}
For the right-hand side of \eqref{eq:appendix-summed-error-energy}, we apply the Cauchy--Schwarz, Young, and Poincar{\'e} inequalities to obtain 
\begin{align}
& 2\dt\sum_{n=2}^m
\left(
  \mathcal R_E^n(\bm e_E^{h,n})
  +\mathcal R_p^n(e_p^{h,n})
\right)
 \le
\dt
\left(1-\frac{L}{\sqrt{\sigma\kappa}}\right)
\sum_{n=2}^m
\mathcal V_h(\bm e_E^{h,n},e_p^{h,n})
+D_m .
\label{eq:appendix-dissipative-residual-sum}
\\
& 2\dt\sum_{n=2}^m
\left(
  \mathcal R_H^n(\bm e_H^{h,n})
  +\bar{\partial}_t\mathcal R_\xi^n(e_\xi^{h,n})
\right)
 \le
\dt\sum_{n=2}^mF_nY_n .
\label{eq:appendix-nondissipative-residual-sum}
\end{align}
Combining \eqref{eq:appendix-summed-error-energy},
\eqref{eq:appendix-dissipative-residual-sum}, and
\eqref{eq:appendix-nondissipative-residual-sum}, we obtain
\begin{align}
Y_m^2 & = \mathcal E^{\mathcal G}_h(e_h^m)
+\dt
\left(1-\frac{L}{\sqrt{\sigma\kappa}}\right)
\sum_{n=2}^m
\mathcal V_h(\bm e_E^{h,n},e_p^{h,n})
\notag \\ & \le
\mathcal E^{\mathcal G}_h(e_h^1)
+D_m+\dt\sum_{n=2}^mF_nY_n = Y_1^2 + D_m +\dt\sum_{n=2}^mF_nY_n .
\label{eq:appendix-summed-energy}
\end{align}
Since \(D_m\) is non-negative and nondecreasing and \(F_n\ge0\),
applying Lemma~\ref{lem:discrete-non-gronwall} to
\eqref{eq:appendix-summed-energy} yields
\begin{equation}
  Y_m^2
  \lesssim
  Y_1^2+D_m+\dt\sum_{n=2}^mF_n^2,
  \qquad 2\le m\le N.
  \label{eq:appendix-non-gronwall-bound}
\end{equation}
The consistency estimate \eqref{eq:appendix-rho-estimate}, together with
the projection estimates
\eqref{eq:hcurl-projection-estimate}, \eqref{eq:l2-h-projection-estimate},
\eqref{eq:stokes-projection-estimate}, and
\eqref{eq:pressure-elliptic-estimate} applied also to first time derivatives, gives
\begin{align*}
D_m & \lesssim \dt\sum_{n=2}^m \left( \lnorm{\bm\rho_E^n}^2 + \lnorm{\rho_\xi^n}^2 + \lnorm{\rho_p^n}^2 + \lnorm{\bm e_E^{I,n}}^2
  +\lnorm{\nabla e_p^{I,n}}^2 \right) 
\notag \\
&\lesssim
\dt^4 \left[\|\partial_t^3\E\|_{L^2(0,T_f;\bm L^2)}^2 + \|\partial_t^3\xi\|_{L^2(0,T_f;L^2)}^2 + \|\partial_t^3p\|_{L^2(0,T_f;L^2)}^2 \right]
\notag \\
& \quad + h^4 \left[ \|\partial_t\E\|_{L^2(0,T_f;\bm H^3)}^2
+ \|\partial_t\xi\|_{L^2(0,T_f;H^2)}^2
+ \|\partial_t p\|_{L^2(0,T_f;H^3)}^2 \right]
\notag \\
& \quad + h^4T_f
\left[
  \|\E\|_{L^\infty(0,T_f;\bm H^3)}^2
  +\|p\|_{L^\infty(0,T_f;H^3)}^2
\right],
\end{align*}
and
\begin{align*}
\dt \sum_{n=2}^m F_n^2
& \lesssim
\dt \sum_{n=2}^m
\left(
  \lnorm{\bm\rho_H^n}^2
  +\lnorm{\curl\bm e_E^{I,n}}^2
  +\lnorm{\bar{\partial}_t e_\xi^{I,n}}^2
  +\lnorm{\bar{\partial}_t e_p^{I,n}}^2
\right)
\notag \\
& \lesssim
  \dt^4\|\partial_t^3\HH\|_{L^2(0,T_f;\bm L^2)}^2
  +h^4\|\partial_t\HH\|_{L^2(0,T_f;\bm H^2)}^2
\notag \\
& \quad + h^4
  \|\partial_t\xi\|_{L^2(0,T_f;H^2)}^2
  + h^4 \|\partial_t p\|_{L^2(0,T_f;H^3)}^2
+
h^4 T_f \|\E\|_{L^\infty(0,T_f;\bm H^3)}^2.
\end{align*}
It follows from \eqref{eq:appendix-non-gronwall-bound} that
\begin{align}
& \max_{2\le n\le N}
  \mathcal E^{\mathcal G}_h(e_h^n)
+\dt\left(1-\frac{L}{\sqrt{\sigma\kappa}}\right)
\sum_{n=2}^N
\mathcal V_h(\bm e_E^{h,n},e_p^{h,n})
\lesssim \mathcal E^{\mathcal G}_h(e_h^1) + 
(h^2+\dt^2)^2.
\label{eq:appendix-discrete-error-estimate}
\end{align}
The discrete inf--sup condition and
\eqref{eq:appendix-discrete-u} imply
\begin{align}
\tilde\beta\lnorm{e_\xi^{h,n}}
&\le
\sup_{\bm0\ne\vv_h\in\Uh}
\frac{\ip{e_\xi^{h,n}}{\diver\vv_h}}{\hnorm{\vv_h}}
=
\sup_{\bm0\ne\vv_h\in\Uh}
\frac{G\ip{\nabla\bm e_u^{h,n}}{\nabla\vv_h}}
     {\hnorm{\vv_h}}
\le
G\lnorm{\nabla\bm e_u^{h,n}}.
\label{eq:appendix-xi-infsup}
\end{align}
Combining \eqref{eq:appendix-discrete-error-estimate} with
\eqref{eq:appendix-xi-infsup}, and then applying the projection estimates and
the triangle inequality to the error decomposition, we obtain the total error estimate. Recalling the $\mathcal G$-norm equivalence \eqref{eq:G-norm-equivalence} and the starting-value assumption, we obtain \eqref{eq:monolithic-bdf2-error}.
\end{proof}


\subsection{Iterative decoupling approximation}

We now describe the iterative decoupling approximation for the monolithic system at the time level $t_n$. Throughout the iterative scheme, we use the notation \((\E_h^{n,i},\HH_h^{n,i},\uu_h^{n,i},\xi_h^{n,i},p_h^{n,i})\) for the corresponding discrete variables, where the superscripts \(n\) and \(i\) denote the time level \(t_n\) and the iteration index, respectively.
Since the BDF2 method is used, the decoupled scheme requires the already computed solutions at the two previous time levels. In addition, the iteration is initialized by the solution from the previous time level:
\[
  \begin{aligned}
  \E_h^{n,0}&=\E_h^{n-1},&
  \HH_h^{n,0}&=\HH_h^{n-1},&
  \uu_h^{n,0}&=\uu_h^{n-1},&
  \xi_h^{n,0}&=\xi_h^{n-1},&
  p_h^{n,0}&=p_h^{n-1}.
  \end{aligned}
\]
For a generic iterate \(\chi_h^{n,i}\), we define the BDF2 derivative as follows.
\[
  \bar{\partial}_t \chi_h^{n,i}
  :=
  \frac{3\chi_h^{n,i}-4\chi_h^{n-1}+\chi_h^{n-2}}{2\dt}.
\]
Each iteration performs two block solves in sequence: first the electromagnetic solve, followed by the poroelastic solve.

\paragraph{Step 1: Electromagnetic subproblem.}
Given the pressure $p_h^{n,i-1}$ from the previous iteration, find
$(\E_h^{n,i},\HH_h^{n,i})\in\Eh\times\Hh$ such that, for all
$(\DD_h,\BB_h)\in\Eh\times\Hh$,
\begin{subequations}
\begin{align}
  \epsilon\ip{\bar{\partial}_t\E_h^{n,i}}{\DD_h}
  +\sigma\ip{\E_h^{n,i}}{\DD_h}
  -\ip{\HH_h^{n,i}}{\curl\DD_h}
  &=
  L\ip{\nabla p_h^{n,i-1}}{\DD_h}
  +\ip{\hh^n}{\DD_h},
  \label{eq:iter-e}\\
  \mu\ip{\bar{\partial}_t\HH_h^{n,i}}{\BB_h}
  +\ip{\curl\E_h^{n,i}}{\BB_h}
  &=0.
  \label{eq:iter-h}
\end{align}
\end{subequations}

\paragraph{Step 2: Poroelastic subproblem.}
With the updated electric field $\E_h^{n,i}$, find
$(\uu_h^{n,i},\xi_h^{n,i},p_h^{n,i})\in\Uh\times\Wh\times\Ph$ such that, for all
$(\vv_h,w_h,q_h)\in\Uh\times\Wh\times\Ph$,
\begin{subequations}
\begin{align}
  G\ip{\nabla\uu_h^{n,i}}{\nabla\vv_h}
  -\ip{\xi_h^{n,i}}{\diver\vv_h}
  &=
  \ip{\ff^n}{\vv_h},
  \label{eq:iter-u}\\
  \ip{\diver\uu_h^{n,i}}{w_h}
  +\lambda^{-1}\ip{\xi_h^{n,i}}{w_h}
  -\alpha\lambda^{-1}\ip{p_h^{n,i}}{w_h}
  &=0,
  \label{eq:iter-xi}\\
  (c_0+\alpha^2\lambda^{-1})
  \ip{\bar{\partial}_t p_h^{n,i}}{q_h}
  -\alpha\lambda^{-1}
  \ip{\bar{\partial}_t\xi_h^{n,i}}{q_h}
  +\kappa\ip{\nabla p_h^{n,i}}{\nabla q_h}
  &=
  L\ip{\E_h^{n,i}}{\nabla q_h}
  +\ip{g^n}{q_h}.
  \label{eq:iter-p}
\end{align}
\end{subequations}
In the computations, the iteration is terminated when either a prescribed maximum number of iterations is reached or the relative change between two successive iterates falls below a chosen tolerance. Upon termination at iteration \(i=I\), we accept the final iterate as the numerical solution at time \(t_n\), namely,
\[
  (\E_h^n,\HH_h^n,\uu_h^n,\xi_h^n,p_h^n)
  :=
  (\E_h^{n,I},\HH_h^{n,I},\uu_h^{n,I},\xi_h^{n,I},p_h^{n,I}),
\]
and proceed to the next time level. The following result establishes the geometric convergence of the decoupled iterates to the monolithic BDF2 solution.

\begin{theorem}
\label{thm:iteration-contraction} 
Let $\{(\E_h^{n,i},\HH_h^{n,i},\uu_h^{n,i},\xi_h^{n,i},p_h^{n,i})\}_{i \geq 1} \subset \mathbb S_h$
be the sequence generated by \eqref{eq:iter-e}--\eqref{eq:iter-p}, and let
$(\E_h^n,\HH_h^n,\uu_h^n,\xi_h^n,p_h^n)\in \mathbb S_h$
be the solution of the monolithic problem
\eqref{eq:mono-e}--\eqref{eq:mono-p}.
Denote the iterative errors by
\[
(\bm e_E^{n,i},\bm e_H^{n,i},\bm e_u^{n,i},e_\xi^{n,i},e_p^{n,i})
:=
(
\E_h^{n,i}-\E_h^n,
\HH_h^{n,i}-\HH_h^n,
\uu_h^{n,i}-\uu_h^n,
\xi_h^{n,i}-\xi_h^n,
p_h^{n,i}-p_h^n
)
\]
Then, for all $i \geq 1$,
\begin{equation}
  \lnorm{\nabla e_p^{n,i}}
  \le
  \frac{L^2}
  {\kappa (\sigma+ \tfrac{ 3\epsilon}{2\dt} )} \lnorm{\nabla e_p^{n,i-1}},
  \label{eq:p-error-contraction}
\end{equation}
where the multiplier is strictly smaller than one, since $L<\sqrt{\sigma \kappa}$. Consequently the iteration is contractive, and 
$(\lnorm{\bm e_E^{n,i}},
\lnorm{\bm e_H^{n,i}},
\hnorm{\bm e_u^{n,i}},
\lnorm{e_\xi^{n,i}},
\hnorm{e_p^{n,i}}) \to \bm{0}$ geometrically as $i \to \infty$.
\end{theorem}

\begin{proof}
We divide the proof into three steps.

\medskip
\noindent
\textbf{1. Electromagnetic error equations.}
Subtracting \eqref{eq:mono-e}--\eqref{eq:mono-h} from
\eqref{eq:iter-e}--\eqref{eq:iter-h} gives
\begin{align}
  (\sigma+\tfrac{3\epsilon}{2\dt})
  \ip{\bm e_E^{n,i}}{\DD_h}
  -\ip{\bm e_H^{n,i}}{\curl\DD_h}
  &=
  L\ip{\nabla e_p^{n,i-1}}{\DD_h},
  \label{eq:em-error-1}\\
  \tfrac{3\mu}{2\dt}
  \ip{\bm e_H^{n,i}}{\BB_h}
  +\ip{\curl\bm e_E^{n,i}}{\BB_h}
  &=0.
  \label{eq:em-error-2}
\end{align}
Taking \((\DD_h,\BB_h)=(\bm{e}_E^{n,i},\bm{e}_H^{n,i})\) and adding \eqref{eq:em-error-1} and \eqref{eq:em-error-2}, we obtain
\begin{align}
  (\sigma+\tfrac{3\epsilon}{2\dt})
  \lnorm{\bm e_E^{n,i}}^2
  +\tfrac{3\mu}{2\dt}\lnorm{\bm e_H^{n,i}}^2
  =
  L\ip{\nabla e_p^{n,i-1}}{\bm e_E^{n,i}}.
  \label{eq:em-error-energy}
\end{align}
Applying the Cauchy--Schwarz and Young inequalities to \eqref{eq:em-error-energy} implies
\begin{align}
  & (\sigma+\tfrac{3\epsilon}{2\dt} )
  \lnorm{\bm e_E^{n,i}}^2
  +\tfrac{3\mu}{2\dt}\lnorm{\bm e_H^{n,i}}^2
   \le
  \frac{L^2}
  {2(\sigma+\tfrac{3\epsilon}{2\dt})}
  \lnorm{\nabla e_p^{n,i-1}}^2 + \frac{1}{2} (\sigma+\tfrac{3\epsilon}{2\dt} )
  \lnorm{\bm e_E^{n,i}}^2.
  \label{eq:em-error-bound}
\end{align}
Dropping the non-negative term, we derive
\begin{equation}
  \lnorm{\bm e_E^{n,i}}
  \le
  \frac{L}
  {\sigma+\tfrac{3\epsilon}{2\dt}}
  \lnorm{\nabla e_p^{n,i-1}}.
  \label{eq:E-error-bound}
\end{equation}

\medskip
\noindent
\textbf{2. Poroelastic error equations.}
Subtracting \eqref{eq:mono-u}--\eqref{eq:mono-p} from
\eqref{eq:iter-u}--\eqref{eq:iter-p} gives
\begin{align}
  G\ip{\nabla\bm e_u^{n,i}}{\nabla\vv_h}
  -\ip{e_\xi^{n,i}}{\diver\vv_h}
  &=0,
  \label{eq:poro-error-u}\\
  \ip{\diver\bm e_u^{n,i}}{w_h}
  +\lambda^{-1}\ip{e_\xi^{n,i}}{w_h}
  -\alpha\lambda^{-1}\ip{e_p^{n,i}}{w_h}
  &=0,
  \label{eq:poro-error-xi}\\
  \tfrac{3}{2\dt}
  (c_0+\alpha^2\lambda^{-1})
  \ip{e_p^{n,i}}{q_h}
  -\alpha\lambda^{-1} \tfrac{3}{2\dt}
  \ip{e_\xi^{n,i}}{q_h}
  +\kappa\ip{\nabla e_p^{n,i}}{\nabla q_h}
  &=
  L\ip{\bm e_E^{n,i}}{\nabla q_h}.
  \label{eq:poro-error-p}
\end{align}
Choosing \((\vv_h,w_h,q_h)=(\bm{e}_u^{n,i},e_\xi^{n,i}, \tfrac{2\dt}{3} e_p^{n,i})\) and 
summing \eqref{eq:poro-error-u}--\eqref{eq:poro-error-p}, we obtain
\begin{align}
  &G\lnorm{\nabla\bm e_u^{n,i}}^2
  +\lambda^{-1}
  \lnorm{e_\xi^{n,i}-\alpha e_p^{n,i}}^2
  +c_0\lnorm{e_p^{n,i}}^2
  +\tfrac{2\dt\kappa}{3}\lnorm{\nabla e_p^{n,i}}^2 =
  \tfrac{2\dt L}{3}
  \ip{\bm e_E^{n,i}}{\nabla e_p^{n,i}}.
  \label{eq:poro-error-energy}
\end{align}
Applying the Cauchy--Schwarz and Young inequalities to \eqref{eq:poro-error-energy} implies
\begin{align}
  & G\lnorm{\nabla\bm e_u^{n,i}}^2
  +\lambda^{-1}
  \lnorm{e_\xi^{n,i}-\alpha e_p^{n,i}}^2
  +c_0\lnorm{e_p^{n,i}}^2
  +\tfrac{2\dt\kappa}{3}\lnorm{\nabla e_p^{n,i}}^2 
  \notag \\
  & \leq 
  \tfrac{\dt L^2}{3\kappa}
  \lnorm{\bm e_E^{n,i}}^2 + \tfrac{\dt\kappa}{3}\lnorm{\nabla e_p^{n,i}}^2.
  \label{eq:poro-error-bound}
\end{align}
Dropping the non-negative terms gives
\begin{equation}
  \lnorm{\nabla e_p^{n,i}}
  \le
  \frac{L}{\kappa}\lnorm{\bm e_E^{n,i}}.
  \label{eq:p-E-error-bound}
\end{equation}

\medskip
\noindent
\textbf{3. Contraction and convergence.}
Combining \eqref{eq:E-error-bound} and \eqref{eq:p-E-error-bound} yields
\begin{equation}
  \lnorm{\nabla e_p^{n,i}}
  \le
  \frac{L}{\kappa}\lnorm{\bm e_E^{n,i}}\le
  \frac{L^2}
  {\kappa(\sigma+\tfrac{3\epsilon}{2\dt})}
  \lnorm{\nabla e_p^{n,i-1}},
  \label{eq:p-error-geometric}
\end{equation}
which proves \eqref{eq:p-error-contraction}. 
The discrete inf--sup condition \eqref{disc:infsup} and
\eqref{eq:poro-error-u} further imply
\begin{equation}
  \tilde\beta\lnorm{e_\xi^{n,i}}
  \le
  \sup_{\bm 0\ne\vv_h\in\Uh}
  \frac{\ip{e_\xi^{n,i}}{\diver\vv_h}}{\hnorm{\vv_h}}
  =
  \sup_{\bm 0\ne\vv_h\in\Uh}
  \frac{G\ip{\nabla\bm e_u^{n,i}}{\nabla\vv_h}}{\hnorm{\vv_h}}
  \le
  G\lnorm{\nabla\bm e_u^{n,i}}. \label{eq:iter-infsup}
\end{equation}
Substituting \eqref{eq:p-error-geometric} into \eqref{eq:em-error-bound} and
\eqref{eq:poro-error-bound} together with the Poincar{\'e} inequality and \eqref{eq:iter-infsup} yields the
corresponding geometric decay of \(\bm e_E^{n,i}\), \(\bm e_H^{n,i}\),
\(\nabla\bm e_u^{n,i}\), \(e_\xi^{n,i}\), and \(e_p^{n,i}\).
This proves the geometric convergence of the iterative scheme to the monolithic solution.
\end{proof}

Together with Theorem~\ref{thm:monolithic-bdf2-error}, this contraction result shows that the iterative solution retains the same finite element accuracy as the monolithic BDF2 solution once the iterative error is reduced below the discretization error.  
Notice also that the contraction factor in \eqref{eq:p-error-contraction} depends on the time step \(\dt\); in particular, for fixed physical parameters, smaller time steps lead to a smaller sufficient contraction factor.

\subsection{Reduced form of the electromagnetic iteration}
\label{sec:reduced-em-iteration}

The electromagnetic part of the iterative decoupling admits an exact algebraic reduction. This reduction both simplifies the implementation and reveals the structure of the pressure-to-electric-field coupling. The key observation is that the discrete pressure gradient is represented exactly in the electric-field space. Indeed, let $p_h\in\Ph$. Since $p_h|_K\in P_2(K)$ for any $K\in\mathcal T_h$, we have $\nabla p_h|_K\in\bm P_1(K)\subset\bm R_2(K)$. 
Moreover, \(p_h\) is globally \(H^1\)-conforming, and hence its trace is continuous across interior faces, which leads to $\nabla p_h\in \bm H(\mathrm{curl};\OmegaD)$. Since $p_h\in H_0^1(\OmegaD)$, its trace is constant on each boundary face,
and therefore the tangential derivative vanishes on $\partial\OmegaD$, i.e., \((\nabla p_h)\times\bn=\bm0\) on \(\partial\OmegaD\). Consequently, \(\nabla p_h\in\Eh\), \(\curl\nabla p_h=\bm0\), and the pressure contribution in the electromagnetic subproblem is an exact discrete gradient field in the N\'ed\'elec space.

To make this observation precise, we subtract \eqref{eq:iter-e}--\eqref{eq:iter-h} at the first iteration from the same equations at iteration \(i>1\). We then obtain, for all \((\DD_h,\BB_h)\in\Eh\times\Hh\),
\begin{align}
  (\sigma+\tfrac{3\epsilon}{2\dt})
  \ip{\E_h^{n,i}-\E_h^{n,1}}{\DD_h}
  -\ip{\HH_h^{n,i}-\HH_h^{n,1}}{\curl\DD_h}
  &=
  L\ip{\nabla(p_h^{n,i-1}-p_h^{n,0})}{\DD_h},
  \label{eq:reduced-em-difference-e}
  \\
  \tfrac{3\mu}{2\dt}
  \ip{\HH_h^{n,i}-\HH_h^{n,1}}{\BB_h}
  +\ip{\curl(\E_h^{n,i}-\E_h^{n,1})}{\BB_h}
  &=0.
  \label{eq:reduced-em-difference-h}
\end{align}
Since \(\nabla(p_h^{n,i-1}-p_h^{n,0})\in\Eh\) and
\(\curl\nabla(p_h^{n,i-1}-p_h^{n,0})=0\), testing
\eqref{eq:reduced-em-difference-e} and \eqref{eq:reduced-em-difference-h} with
\[
  \DD_h
  =
  \E_h^{n,i}-\E_h^{n,1}
  -
  \frac{L}{\sigma+\frac{3\epsilon}{2\dt}}
  \nabla(p_h^{n,i-1}-p_h^{n,0}), \qquad \BB_h=\HH_h^{n,i}-\HH_h^{n,1},
\]
respectively, we derive
\[
  (\sigma+\tfrac{3\epsilon}{2\dt})
  \bigg\|
  \E_h^{n,i}-\E_h^{n,1}
  -
  \frac{L}{\sigma+\frac{3\epsilon}{2\dt}}
  \nabla(p_h^{n,i-1}-p_h^{n,0})
  \bigg\|_{L^2}^2
  +
  \tfrac{3\mu}{2\dt}
  \|\HH_h^{n,i}-\HH_h^{n,1}\|_{L^2}^2
  =0 .
\]
The same identity holds trivially for \(i=1\). Consequently, for all \(i\ge1\),
\begin{equation}
  \E_h^{n,i}
  =
  \E_h^{n,1}
  +
  \frac{L}{\sigma+\frac{3\epsilon}{2\dt}}
  \nabla(p_h^{n,i-1}-p_h^{n,0}),
  \qquad
  \HH_h^{n,i}=\HH_h^{n,1}.
  \label{eq:reduced-em-explicit-update}
\end{equation}

Formula \eqref{eq:reduced-em-explicit-update} shows that, after the first electromagnetic solve at time \(t_n\), the magnetic field remains fixed, whereas the electric field is updated by a curl-free pressure-gradient correction. Substituting this representation into the mass equation yields the pressure-feedback term in \eqref{eq:riter-p}. Hence, no additional electromagnetic solve is required during the subsequent iterations. This leads to the following reduced scheme, which is algebraically equivalent to the original iterative decoupling method.

\paragraph{Initial electromagnetic solve.}
Given the initial pressure iterate \(p_h^{n,0}\), find
\((\E_h^{n,1},\HH_h^{n,1})\in\Eh\times\Hh\) such that, for all
\((\DD_h,\BB_h)\in\Eh\times\Hh\),
\begin{subequations}
\begin{align}
  \epsilon\ip{\bar{\partial}_t\E_h^{n,1}}{\DD_h}
  +\sigma\ip{\E_h^{n,1}}{\DD_h}
  -\ip{\HH_h^{n,1}}{\curl\DD_h}
  &=
  L\ip{\nabla p_h^{n,0}}{\DD_h}
  +\ip{\hh^n}{\DD_h},
  \label{eq:riter-e}\\
  \mu\ip{\bar{\partial}_t\HH_h^{n,1}}{\BB_h}
  +\ip{\curl\E_h^{n,1}}{\BB_h}
  &=0.
  \label{eq:riter-h}
\end{align}
\end{subequations}

\paragraph{Reduced poroelastic iteration.}
For \(i\ge1\), compute
\((\uu_h^{n,i},\xi_h^{n,i},p_h^{n,i})\in\Uh\times\Wh\times\Ph\) from
\begin{subequations}
\begin{align}
  G\ip{\nabla\uu_h^{n,i}}{\nabla\vv_h}
  -\ip{\xi_h^{n,i}}{\diver\vv_h}
  &=
  \ip{\ff^n}{\vv_h},
  \label{eq:riter-u}\\
   \ip{\diver\uu_h^{n,i}}{w_h}
  +\lambda^{-1}\ip{\xi_h^{n,i}}{w_h}
  -\alpha\lambda^{-1}\ip{p_h^{n,i}}{w_h}
  &=0,
  \label{eq:riter-xi}\\
  (c_0+\alpha^2\lambda^{-1})
  \ip{\bar{\partial}_t p_h^{n,i}}{q_h}
  -\alpha\lambda^{-1}
  \ip{\bar{\partial}_t\xi_h^{n,i}}{q_h}
  +\kappa\ip{\nabla p_h^{n,i}}{\nabla q_h} 
  &=
  L\ip{\E_h^{n,1}}{\nabla q_h}   \label{eq:riter-p}\\ 
  +
  \frac{L^2}{\sigma+\frac{3\epsilon}{2\dt}}
  &\ip{\nabla p_h^{n,i-1}-\nabla p_h^{n,0}}{\nabla q_h}
  +\ip{g^n}{q_h}, \notag
\end{align}
\end{subequations}
for all \((\vv_h,w_h,q_h)\in\Uh\times\Wh\times\Ph\). The 
electromagnetic fields are then recovered from
\eqref{eq:reduced-em-explicit-update}.
\\

Since the reduced scheme is algebraically equivalent to the original iteration, the convergence result of Theorem~\ref{thm:iteration-contraction} applies unchanged.

\begin{remark}[Reduced electromagnetic solve]
\label{rem:reduced-em-solve}
The initial electromagnetic solve in \eqref{eq:riter-e}--\eqref{eq:riter-h} can also be reduced to a single electric-field problem. Indeed, since \(\curl\Eh\subset\Hh\), equation \eqref{eq:riter-h} gives
\[
  \HH_h^{n,1}
  =
  \frac{4}{3}\HH_h^{n-1}-\frac{1}{3}\HH_h^{n-2}
  -
  \frac{2\dt}{3\mu}\curl\E_h^{n,1}.
\]
Substituting this expression into \eqref{eq:riter-e}, one obtains the curl--curl problem for \(\E_h^{n,1}\):
\[
\begin{aligned}
&(\sigma+\tfrac{3\epsilon}{2\dt})
\ip{\E_h^{n,1}}{\DD_h}
+
\tfrac{2\dt}{3\mu}
\ip{\curl\E_h^{n,1}}{\curl\DD_h}
\\
& =
L\ip{\nabla p_h^{n,0}}{\DD_h}
+
\ip{\hh^n}{\DD_h}
+
\tfrac{\epsilon}{2\dt}
\ip{4\E_h^{n-1}-\E_h^{n-2}}{\DD_h}
+
\ip{\tfrac{4}{3}\HH_h^{n-1}-\tfrac{1}{3}\HH_h^{n-2}}{\curl\DD_h}.
\end{aligned}
\]
Thus, the electromagnetic initialization at each BDF2 time level requires only one curl--curl solve for \(\E_h^{n,1}\), after which \(\HH_h^{n,1}\) is recovered explicitly from the formula above.
\end{remark}

\begin{remark}[Nested poroelastic decoupling]
In the present strategy, the poroelasticity block \eqref{eq:riter-u}--\eqref{eq:riter-p} is solved as one coupled subproblem after the electromagnetic update. This block may itself be further decoupled by existing algorithms for Biot's model \cite{gu2023iterative, cai2023some}. Such a choice would lead to a nested decoupling strategy, with the outer electromagnetic--poroelastic iteration combined with an inner poroelastic solver. The analysis of this nested strategy is beyond the scope of the present work.
\end{remark}

\section{Numerical experiments}
\label{sec:4}
This section presents numerical experiments for the proposed five-field BDF2 finite element method with the finite element spaces introduced in Section~\ref{STD}. 
The experiments include space--time convergence tests, parameter-dependent simulations, and comparisons between the original and reduced iterative formulations. A three-dimensional footing benchmark is also considered to demonstrate the locking-free performance of the method in the nearly incompressible regime. All computations are carried out using FEniCSx \cite{BarattaEtal2023} and multiphenicsx \cite{multiphenicsx}.

Unless otherwise stated, the linear systems are solved by sparse LU factorization with MUMPS through PETSc. For the finer-mesh convergence tests in Tables \ref{tab:spatial}–\ref{tab:parameter-robust}, we instead use PETSc’s GMRES implementation with ILU preconditioning. At each time level, the fixed-point iteration is terminated when the maximum relative update over all solution components satisfies
\[
  \delta_i^n
  :=
  \max_{\zeta_h\in\mathcal Z_h}
  \frac{
    \|\zeta_h^{n,i}-\zeta_h^{n,i-1}\|_{L^2}
  }{
    \|\zeta_h^{n,i}\|_{L^2}+\varepsilon_{\rm abs}
  }
  \le 10^{-10},
  \qquad
  \mathcal Z_h:=\{\E_h,\HH_h,\uu_h,\xi_h,p_h\},
\]
where \(\varepsilon_{\rm abs}>0\) is a small safeguard to avoid division by zero.

\subsection{Verification of convergence}
We consider a test with a manufactured solution on $\OmegaD = (0,1)^3$. Consistent with the mixed boundary setting for
the displacement, we take
$\Gamma_N^u:=\{(1,y,z):0\le y,z\le1\}$, 
$\Gamma_D^u:=\partial\OmegaD\setminus\Gamma_N^u$. 
We construct the analytical solution as follows:
\begin{align*}
  \E & =  \begin{bmatrix}
    \sin{(\pi x)} \sin{(\pi y)} \sin{(\pi z)} \\
    \sin{(\pi x)} \sin{(\pi y)} \sin{(\pi z)} \\
    \sin{(\pi x)} \sin{(\pi y)} \sin{(\pi z)} \\
    \end{bmatrix} e^t, 
  \\
  \HH & =   \begin{bmatrix}
    \sin{(\pi x)} \cos{(\pi y)} \sin{(\pi z)} - \sin{(\pi x)} \sin{(\pi y)} \cos{(\pi z)} \\
    \sin{(\pi x)} \sin{(\pi y)} \cos{(\pi z)} - \cos{(\pi x)} \sin{(\pi y)} \sin{(\pi z)} \\
    \cos{(\pi x)} \sin{(\pi y)} \sin{(\pi z)} - \sin{(\pi x)} \cos{(\pi y)} \sin{(\pi z)} \\
    \end{bmatrix} \left( -\frac{\pi}{\mu} \right)e^t, 
  \\
  \uu & =  \begin{bmatrix}
    \sin(\pi x)\cos(\pi y)\cos(\pi z)\\
    \cos(\pi x)\sin(\pi y)\cos(\pi z)\\
    \cos(\pi x)\cos(\pi y)\sin(\pi z)
  \end{bmatrix} xyz(1-x)^2(1-y)(1-z) e^t,
  \\
  p & = \sin{(\pi x)} \sin{(\pi y)} \sin{(\pi z)} e^t,
\end{align*}
The source terms \(\hh,\ff,g\) are obtained by substituting this exact
solution into the five-field equations \eqref{eq:five-e}--\eqref{eq:five-p}. Unless otherwise stated, the material parameters are
\begin{align}
  \epsilon = \sigma = \mu = \alpha = G = \lambda = c_0 = \kappa = 1.0, \quad L=0.5.
\label{baseline}
\end{align}

\paragraph{Temporal convergence} 
For the temporal test, the spatial mesh and final time \(T_f=1.0\) are fixed, and only the time step is refined for the monolithic approximation \eqref{eq:mono}.  Consequently, the final-time errors in Table~\ref{tab:time} mainly measure the time discretization error.  The convergence rates are close to the expected second order for all components, consistent with the BDF2 method.

\begin{table}[htbp]
\centering
\caption{Temporal convergence for the five-field monolithic BDF2 method.}
\label{tab:time}
\resizebox{\textwidth}{!}{%
\begin{tabular}{cccccc}
\toprule
\(\dt\)
& \(\lnorm{\E_{h,\dt}^N-\E_{h,\dt/2}^N}\)
& \(\lnorm{\HH_{h,\dt}^N-\HH_{h,\dt/2}^N}\)
& \(\hnorm{\uu_{h,\dt}^N-\uu_{h,\dt/2}^N}\)
& \(\lnorm{\xi_{h,\dt}^N-\xi_{h,\dt/2}^N}\)
& \(\lnorm{p_{h,\dt}^N-p_{h,\dt/2}^N}\) \\
\midrule
\(1/40\)
& \(3.408{\times}10^{-4}\)
& \(3.141{\times}10^{-4}\)
& \(3.353{\times}10^{-6}\)
& \(4.956{\times}10^{-6}\)
& \(7.504{\times}10^{-6}\) \\
\(1/80\)
& \(9.323{\times}10^{-5}\) (1.87)
& \(7.007{\times}10^{-5}\) (2.16)
& \(8.431{\times}10^{-7}\) (1.99)
& \(1.248{\times}10^{-6}\) (1.99)
& \(1.888{\times}10^{-6}\) (1.99) \\
\(1/160\)
& \(2.393{\times}10^{-5}\) (1.96)
& \(1.649{\times}10^{-5}\) (2.09)
& \(2.114{\times}10^{-7}\) (2.00)
& \(3.131{\times}10^{-7}\) (2.00)
& \(4.735{\times}10^{-7}\) (2.00) \\
\(1/320\)
& \(6.036{\times}10^{-6}\) (1.99)
& \(4.002{\times}10^{-6}\) (2.04)
& \(5.293{\times}10^{-8}\) (2.00)
& \(7.840{\times}10^{-8}\) (2.00)
& \(1.186{\times}10^{-7}\) (2.00) \\
\bottomrule
\end{tabular}
}
\end{table}

\paragraph{Spatial convergence}

For the spatial test, we fix the final time \(T_f=10^{-3}\) and take \(\dt=T_f/8\).  
The computations are performed on uniform tetrahedral refinements of the unit cube with mesh size \(h\).  This small final time and time step keep the BDF2 error below the spatial error for the electric field.  The convergence rates in Table~\ref{tab:spatial} are computed with respect to \(h\).  The electric field, magnetic field, total pressure, and displacement \(H^1\)-error display the expected second-order behavior.  The pressure \(L^2\)-error is approximately third order, as is typical for the quadratic conforming pressure approximation under the present smooth manufactured solution.

\begin{table}[htbp]
\centering
\caption{Spatial convergence of the five-field monolithic BDF2 method.}
\label{tab:spatial}
\resizebox{\textwidth}{!}{%
\begin{tabular}{cccccc}
\toprule
\(h\)
& \(\lnorm{\E-\E_h}\)
& \(\lnorm{\HH-\HH_h}\)
& \(\hnorm{\uu-\uu_h}\)
& \(\lnorm{\xi-\xi_h}\)
& \(\lnorm{p-p_h}\) \\
\midrule
1/4
& \(5.766{\times}10^{-2}\)
& \(3.047{\times}10^{-1}\)
& \(2.250{\times}10^{-2}\)
& \(4.351{\times}10^{-2}\)
& \(3.835{\times}10^{-2}\) \\
1/8
& \(1.494{\times}10^{-2}\;(1.95)\)
& \(7.882{\times}10^{-2}\;(1.95)\)
& \(5.206{\times}10^{-3}\;(2.11)\)
& \(1.139{\times}10^{-2}\;(1.93)\)
& \(1.118{\times}10^{-2}\;(1.78)\) \\
1/16
& \(3.798{\times}10^{-3}\;(1.98)\)
& \(1.987{\times}10^{-2}\;(1.99)\)
& \(1.236{\times}10^{-3}\;(2.07)\)
& \(2.874{\times}10^{-3}\;(1.99)\)
& \(2.905{\times}10^{-3}\;(1.94)\) \\
1/32
& \(9.824{\times}10^{-4}\;(1.95)\)
& \(4.972{\times}10^{-3}\;(2.00)\)
& \(3.029{\times}10^{-4}\;(2.03)\)
& \(7.199{\times}10^{-4}\;(2.00)\)
& \(7.330{\times}10^{-4}\;(1.99)\) \\
\bottomrule
\end{tabular}
}
\end{table}

\paragraph{Coupled space--time refinement.}

In addition to the separate temporal and spatial tests above, we also consider a coupled refinement path in which the mesh size and time step are refined together.  We take \(T_f=1.0\) and set \(h=\dt=1/M\) for \(M=4,8,16,32\). The convergence rates in Table~\ref{tab:parameter-moderate} are computed with respect to this common refinement parameter. Overall, the results are consistent with the expected second-order space--time convergence, although the displacement and pressure errors exhibit higher observed rates over the range of refinements considered.

\begin{table}[htbp]
\centering
\caption{Coupled space--time convergence test of the five-field monolithic BDF2 method.}
\label{tab:parameter-moderate}
\resizebox{\textwidth}{!}{%
\begin{tabular}{cccccc}
\toprule
\(M\)
& \(\lnorm{\E-\E_h}\)
& \(\lnorm{\HH-\HH_h}\)
& \(\hnorm{\uu-\uu_h}\)
& \(\lnorm{\xi-\xi_h}\)
& \(\lnorm{p-p_h}\) \\
\midrule
4
& \(1.676{\times}10^{-1}\)
& \(2.955{\times}10^{-1}\)
& \(5.011{\times}10^{-2}\)
& \(8.448{\times}10^{-2}\)
& \(1.503{\times}10^{-2}\) \\
8
& \(4.852{\times}10^{-2}\;(1.79)\)
& \(7.345{\times}10^{-2}\;(2.01)\)
& \(7.963{\times}10^{-3}\;(2.65)\)
& \(2.000{\times}10^{-2}\;(2.08)\)
& \(1.867{\times}10^{-3}\;(3.01)\) \\
16
& \(1.270{\times}10^{-2}\;(1.93)\)
& \(1.818{\times}10^{-2}\;(2.01)\)
& \(1.269{\times}10^{-3}\;(2.65)\)
& \(4.873{\times}10^{-3}\;(2.04)\)
& \(2.413{\times}10^{-4}\;(2.95)\) \\
32
& \(3.189{\times}10^{-3}\;(1.99)\)
& \(4.534{\times}10^{-3}\;(2.00)\)
& \(2.356{\times}10^{-4}\;(2.43)\)
& \(1.209{\times}10^{-3}\;(2.01)\)
& \(3.431{\times}10^{-5}\;(2.81)\) \\
\bottomrule
\end{tabular}
}
\end{table}

\paragraph{Parameter robustness.}

We next repeat the coupled refinement test in a more challenging parameter regime.  
Keeping \(\epsilon=\sigma=\mu=\alpha=1.0\), we choose stiff, nearly incompressible mechanics and small permeability:
\[
  G=10^6,\quad \lambda=10^6,\quad c_0=10^{-6},\quad
  \kappa=10^{-2}, \quad
  L=\frac{\sqrt{\sigma\kappa}}{2}=5.0\times10^{-2}.
\]
Thus the coupling coefficient is scaled with the permeability so that the stability condition remains uniformly separated from its limiting value.
To isolate the effect of the material parameters, Table~\ref{tab:parameter-robust} reports the corresponding errors using the same final time and coupled refinement path as in Table~\ref{tab:parameter-moderate}.
The absolute total-pressure error is large in this test because \(\xi=\alpha p-\lambda\diver\uu\) contains the factor \(\lambda=10^6\).  Nevertheless, all variables retain the expected convergence behavior.  Comparing Tables~\ref{tab:parameter-moderate} and \ref{tab:parameter-robust}, we observe that the large bulk modulus, small storage coefficient, and small permeability do not visibly degrade the convergence rates.

\begin{table}[htbp]
\centering
\caption{Parameter-robust convergence test of the five-field monolithic BDF2 method.}
\label{tab:parameter-robust}
\resizebox{\textwidth}{!}{%
\begin{tabular}{cccccc}
\toprule
\(M\)
& \(\lnorm{\E-\E_h}\)
& \(\lnorm{\HH-\HH_h}\)
& \(\hnorm{\uu-\uu_h}\)
& \(\lnorm{\xi-\xi_h}\)
& \(\lnorm{p-p_h}\) \\
\midrule
4
& \(1.676{\times}10^{-1}\)
& \(2.954{\times}10^{-1}\)
& \(1.084{\times}10^{-2}\)
& \(8.304{\times}10^{3}\)
& \(1.793{\times}10^{-2}\) \\
8
& \(4.851{\times}10^{-2}\;(1.79)\)
& \(7.344{\times}10^{-2}\;(2.01)\)
& \(3.024{\times}10^{-3}\;(1.84)\)
& \(2.200{\times}10^{3}\;(1.92)\)
& \(2.653{\times}10^{-3}\;(2.76)\) \\
16
& \(1.270{\times}10^{-2}\;(1.93)\)
& \(1.818{\times}10^{-2}\;(2.01)\)
& \(7.451{\times}10^{-4}\;(2.02)\)
& \(5.098{\times}10^{2}\;(2.11)\)
& \(4.882{\times}10^{-4}\;(2.44)\) \\
32
& \(3.188{\times}10^{-3}\;(1.99)\)
& \(4.534{\times}10^{-3}\;(2.00)\)
& \(1.854{\times}10^{-4}\;(2.01)\)
& \(1.243{\times}10^{2}\;(2.04)\)
& \(1.091{\times}10^{-4}\;(2.16)\) \\
\bottomrule
\end{tabular}
}
\end{table}

\paragraph{Iterative solver performance.}

We finally examine how the convergence of the iterative decoupling depends on the parameters appearing in the contraction factor. In this test, one parameter among \(L\), \(\epsilon\), \(\sigma\), \(\kappa\), \(\dt\), and \(h\), is varied at a time, while the remaining parameters are fixed at the baseline values \eqref{baseline}. The coupling values are always chosen so that \(L<\sqrt{\sigma\kappa}\). For each run, the plotted quantity is \(\lnorm{\nabla e_p^{N,i}}\), where \(e_p^{N,i}=p_h^{N,i}-p_h^N\) denotes the pressure iteration error with respect to the monolithic solution at the final time. This is the quantity directly controlled by \eqref{eq:p-error-contraction}.
The six panels in Figure~\ref{fig:iterative-parameter-sweeps} show the expected geometric decay.  The \(L\)-panel shows that stronger coupling slows the iteration. 
Similarly, increasing \(\epsilon\), \(\sigma\), or \(\kappa\) improves the contraction. The \(\dt\)-panel shows faster decay for smaller time steps, because the denominator \(\sigma+3\epsilon/(2\dt)\) increases as \(\dt\) decreases. The \(h\)-panel shows nearly overlapping curves, which is consistent with the mesh-independent nature of the sufficient contraction factor.

\begin{figure}[h]
\centering
  \includegraphics[width=\textwidth]{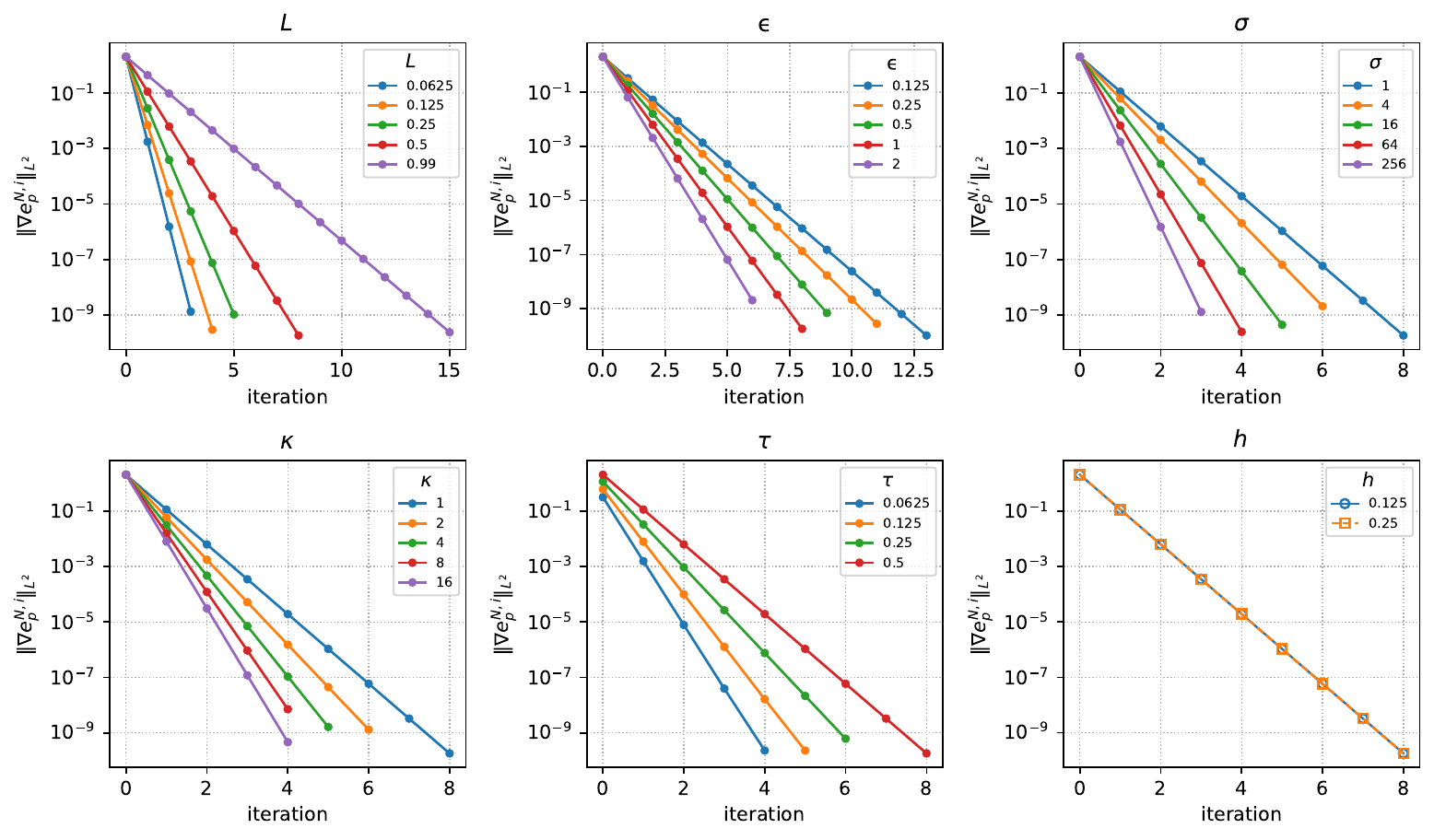}
\caption{Parameter sweeps over \((L,\epsilon,\sigma,\kappa,\dt,h)\) for the iterative decoupling, showing the convergence behavior of the final-time iteration error \(\lnorm{\nabla e_p^{N,i}}\).}
\label{fig:iterative-parameter-sweeps}
\end{figure}

\paragraph{Equivalence and efficiency of the reduced iteration.}
We compare the original iterative decoupling scheme \eqref{eq:iter-e}--\eqref{eq:iter-p} with its reduced electromagnetic form \eqref{eq:riter-e}--\eqref{eq:riter-p} using \(h=1/18\), \(\dt=5.0\times10^{-2}\), \(T_f=1.0\times10^{-1}\), and the parameters in \eqref{baseline}.
Table~\ref{tab:original-reduced-iteration-differences} reports the differences between the corresponding iterates at the final time level.
The results confirm that the two formulations agree up to roundoff error, as predicted by their algebraic equivalence.

\begin{table}[htbp]
\centering
\caption{Differences between the iterates produced by the two iterative decoupling schemes.}
\label{tab:original-reduced-iteration-differences}
\resizebox{\textwidth}{!}{%
\begin{tabular}{cccccc}
\toprule
\(i\)
& \(\lnorm{\E_{h,\mathrm{ori}}^{N,i}-\E_{h,\mathrm{red}}^{N,i}}\)
& \(\lnorm{\HH_{h,\mathrm{ori}}^{N,i}-\HH_{h,\mathrm{red}}^{N,i}}\)
& \(\hnorm{\uu_{h,\mathrm{ori}}^{N,i}-\uu_{h,\mathrm{red}}^{N,i}}\)
& \(\lnorm{\xi_{h,\mathrm{ori}}^{N,i}-\xi_{h,\mathrm{red}}^{N,i}}\)
& \(\lnorm{p_{h,\mathrm{ori}}^{N,i}-p_{h,\mathrm{red}}^{N,i}}\)\\
\midrule
1 & \(2.045{\times}10^{-14}\) & \(1.824{\times}10^{-14}\)
  & \(1.279{\times}10^{-16}\) & \(1.589{\times}10^{-16}\)
  & \(3.116{\times}10^{-16}\) \\
2 & \(2.143{\times}10^{-14}\) & \(1.825{\times}10^{-14}\)
  & \(1.292{\times}10^{-16}\) & \(1.587{\times}10^{-16}\)
  & \(3.126{\times}10^{-16}\) \\
3 & \(2.146{\times}10^{-14}\) & \(1.824{\times}10^{-14}\)
  & \(1.315{\times}10^{-16}\) & \(1.619{\times}10^{-16}\)
  & \(3.245{\times}10^{-16}\) \\
4 & \(2.145{\times}10^{-14}\) & \(1.825{\times}10^{-14}\)
  & \(1.286{\times}10^{-16}\) & \(1.576{\times}10^{-16}\)
  & \(3.117{\times}10^{-16}\) \\
5 & \(2.144{\times}10^{-14}\) & \(1.825{\times}10^{-14}\)
  & \(1.296{\times}10^{-16}\) & \(1.640{\times}10^{-16}\)
  & \(3.194{\times}10^{-16}\) \\
6 & \(2.144{\times}10^{-14}\) & \(1.825{\times}10^{-14}\)
  & \(1.288{\times}10^{-16}\) & \(1.598{\times}10^{-16}\)
  & \(3.120{\times}10^{-16}\) \\
\bottomrule
\end{tabular}
}
\end{table}

We next compare the computational costs of the monolithic, original iterative decoupling, and reduced iterative decoupling schemes. Table~\ref{tab:cpu-cost-comparison} reports the iteration counts, wall-clock times, speedups relative to the monolithic scheme, and absolute differences from the monolithic solution. Both iterative schemes converge in six iterations per time level. The original iterative decoupling scheme achieves speedups of \(2.52\) and \(3.23\), whereas the reduced iterative decoupling scheme increases these values to \(3.40\) and \(4.75\), respectively. Under the stopping criterion specified above, both iterative schemes produce solutions that are numerically indistinguishable from the monolithic solution.

\begin{table}[htbp]
\centering
\caption{Iteration counts, wall-clock times, and speedups of the three schemes on two fine meshes, together with their absolute differences from the monolithic solution.}
\label{tab:cpu-cost-comparison}
\resizebox{\textwidth}{!}{%
\begin{tabular}{cclcccc}
\toprule
\(h\) & dofs & method & avg./max iter. & CPU time (s) & speedup  & abs. diff. \\
\midrule 
&  
& Monolithic approximation
& \(--\)
& \(100.098\)
& \(1.00\)
& \(0.0\) \\
\(1/18\)
& \(860819\)
& Original iterative decoupling
& \(6.00/6\)
& \(39.662\)
& \(2.52\)
& \(3.19{\times}10^{-14}\) \\
& 
& Reduced iterative decoupling
& \(6.00/6\)
& \(29.434\)
& \(3.40\)
& \(3.20{\times}10^{-14}\) \\
\midrule
&  
& Monolithic approximation
& \(--\)
& \(265.958\)
& \(1.00\)
& \(0.0\) \\
\(1/20\)
& \(1177065\)
& Original iterative decoupling
& \(6.00/6\)
& \(82.237\)
& \(3.23\)
& \(3.19{\times}10^{-14}\) \\
& 
& Reduced iterative decoupling
& \(6.00/6\)
& \(55.986\)
& \(4.75\)
& \(3.21{\times}10^{-14}\) \\
\bottomrule
\end{tabular}
}
\end{table}

\subsection{Three-dimensional footing benchmark}

We next test the implementation on the three-dimensional footing problem. The domain is the unit cube \(\OmegaD=(0,1)^3\). The bottom face and the four lateral faces are mechanically clamped, whereas the top face is divided into a loaded region and a traction-free region. Let
\begin{align*}
  & \Gamma_{\rm load}
  =\{(x,y,z):z=1,\ 0.3\le x\le0.7,\ 0.3\le y\le0.7\}, \\ 
  & \Gamma_{\rm free} =\{z=1\}\setminus\overline{\Gamma_{\rm load}}, \\ 
  & \Gamma_D^u = \partial \OmegaD \setminus(\overline{\Gamma_{\rm load}}\cup \overline{\Gamma_{\rm free}}). 
\end{align*}
A vertical traction $t_z$ is applied on \(\Gamma_{\rm load}\). We consider the following boundary conditions.
\begin{align*}
  & (G\nabla\uu-\xi\mathbf I)\bn = -t_z\bm e_3 
      &&\text{on } \Gamma_{\rm load}, \\
  & (G\nabla\uu-\xi\mathbf I)\bn = \bm0 
      &&\text{on } \Gamma_{\rm free}, \\
  & \uu =0 &&\text{on } \Gamma_D^u, \\
  & p =0 &&\text{on } \partial\OmegaD, \\
  & \E\times\bn = \bm0 &&\text{on } \partial\OmegaD.
\end{align*}
Here \(\bm e_3=(0,0,1)^T\). 
All initial fields are zero. 
The mechanical body force $\ff$, and fluid source $g$ are set to zero.
The electromagnetic source is prescribed as the smooth swirl field
\[
  \hh(\bm x)
  =10^{-3}
  \prod_{j=1}^3\sin^2(\pi x_j)
  \begin{bmatrix}
    -(y-1/2)\\ x-1/2\\ z-1/2
  \end{bmatrix}.
\]
For the temporal discretization, we fix $T_f = 0.1$ and use $\dt = 0.01$. For the spatial discretization, we consider the uniform tetrahedral mesh with $h = 1/20$. The physical parameters used in the simulations are summarized in Table \ref{tab:footing-parameters}.
\begin{table}[htbp]
\centering
\caption{Parameters used in the footing problem.}
\label{tab:footing-parameters}
\begin{tabular}{lll}
\toprule
parameter & symbol & value \\
\midrule
electric permittivity & \(\epsilon\) & \(1.0\) \\
electric conductivity & \(\sigma\) & \(1.0\) \\
magnetic permeability & \(\mu\) & \(1.0\) \\
Biot--Willis coefficient & \(\alpha\) & \(1.0\) \\
storage coefficient & \(c_0\) & \(2.0{\times}10^{-3}\) \\
permeability coefficient & \(\kappa\) & \(1.0{\times}10^{-4}\) \\
electrokinetic coupling coefficient & \(L\) & \(5.0{\times}10^{-5}\) \\
shear modulus & \(G\) & \(1.0{\times}10^4\) \\
compressibility modulus & \(\lambda\) & \(4.993{\times}10^6\) \\
vertical traction & \(t_z\) & \(5.0{\times}10^3\) \\
\bottomrule
\end{tabular}
\end{table}



In Figure~\ref{fig:footing-benchmark}, we compare the computed pressure field at the first time step $t = 0.01$ on the center section \(y=0.5\).
The comparison is made between the five-field formulation \eqref{eq:five}, using poroelastic spaces \(({\bm P}_2,P_1,P_2)\) and \(({\bm P}_1,P_1,P_1)\), and the four-field formulation \eqref{eq:four}, using poroelastic spaces \(({\bm P}_2,P_2)\) and \(({\bm P}_1,P_1)\).
The stable five-field discretization \(({\bm P}_2,P_1,P_2)\) produces a smooth pressure profile, whereas the equal-order five-field discretization \(({\bm P}_1,P_1,P_1)\) shows visible oscillations in the pressure field.
The four-field formulation with poroelastic space \(({\bm P}_2,P_2)\) also exhibits pressure oscillations, while the \(({\bm P}_1,P_1)\) result shows more pronounced banded artifacts.  This comparison indicates that the stable five-field choice gives the cleanest pressure response for this nearly incompressible footing test.

\begin{figure}[ht]
\centering
\begin{minipage}{0.24\textwidth}
\centering
\includegraphics[width=\linewidth]{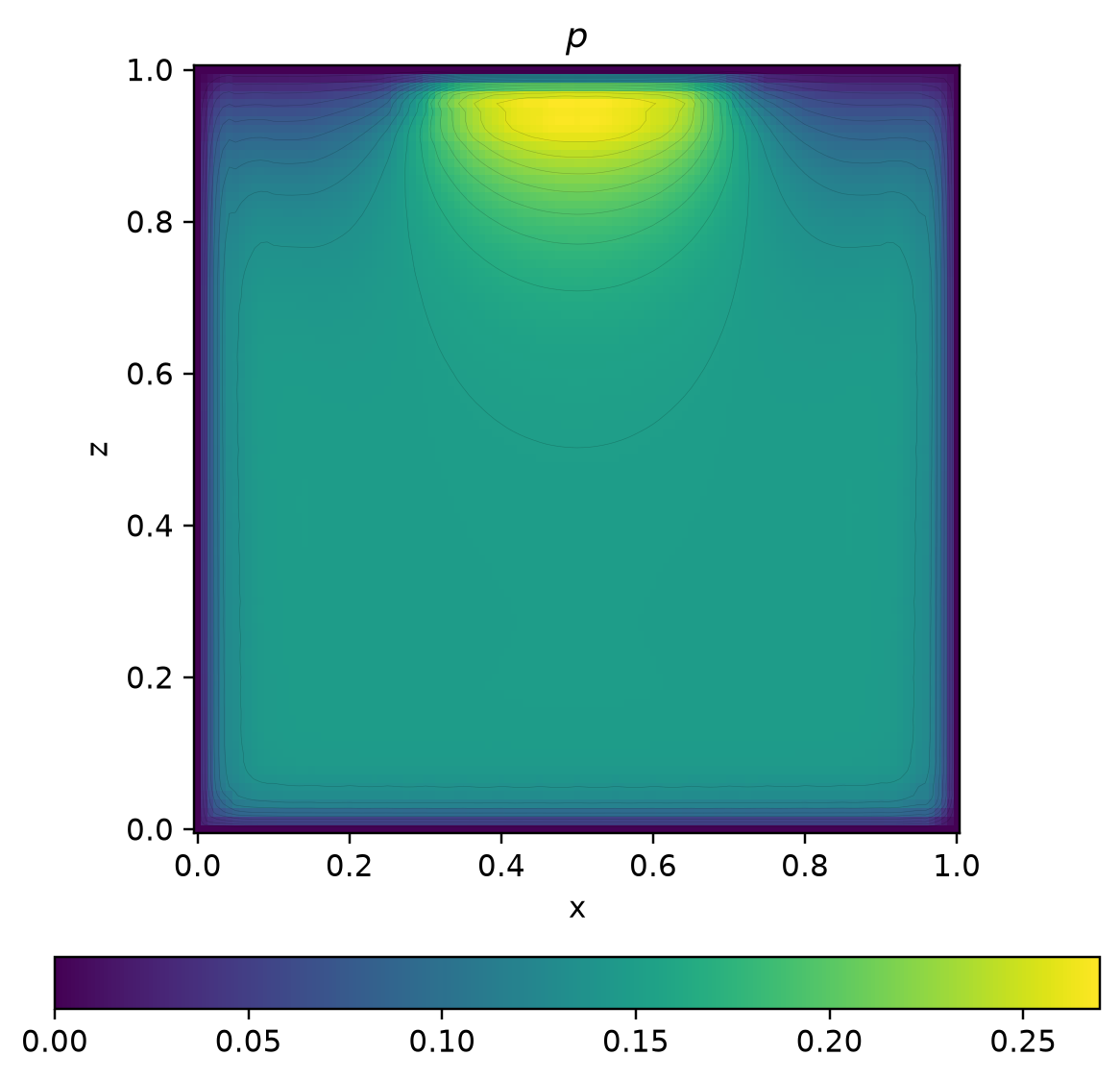}\\[-0.2em]
{\small \(({\bm P}_2,P_1,P_2)\)}
\end{minipage}\hfill
\begin{minipage}{0.24\textwidth}
\centering
\includegraphics[width=\linewidth]{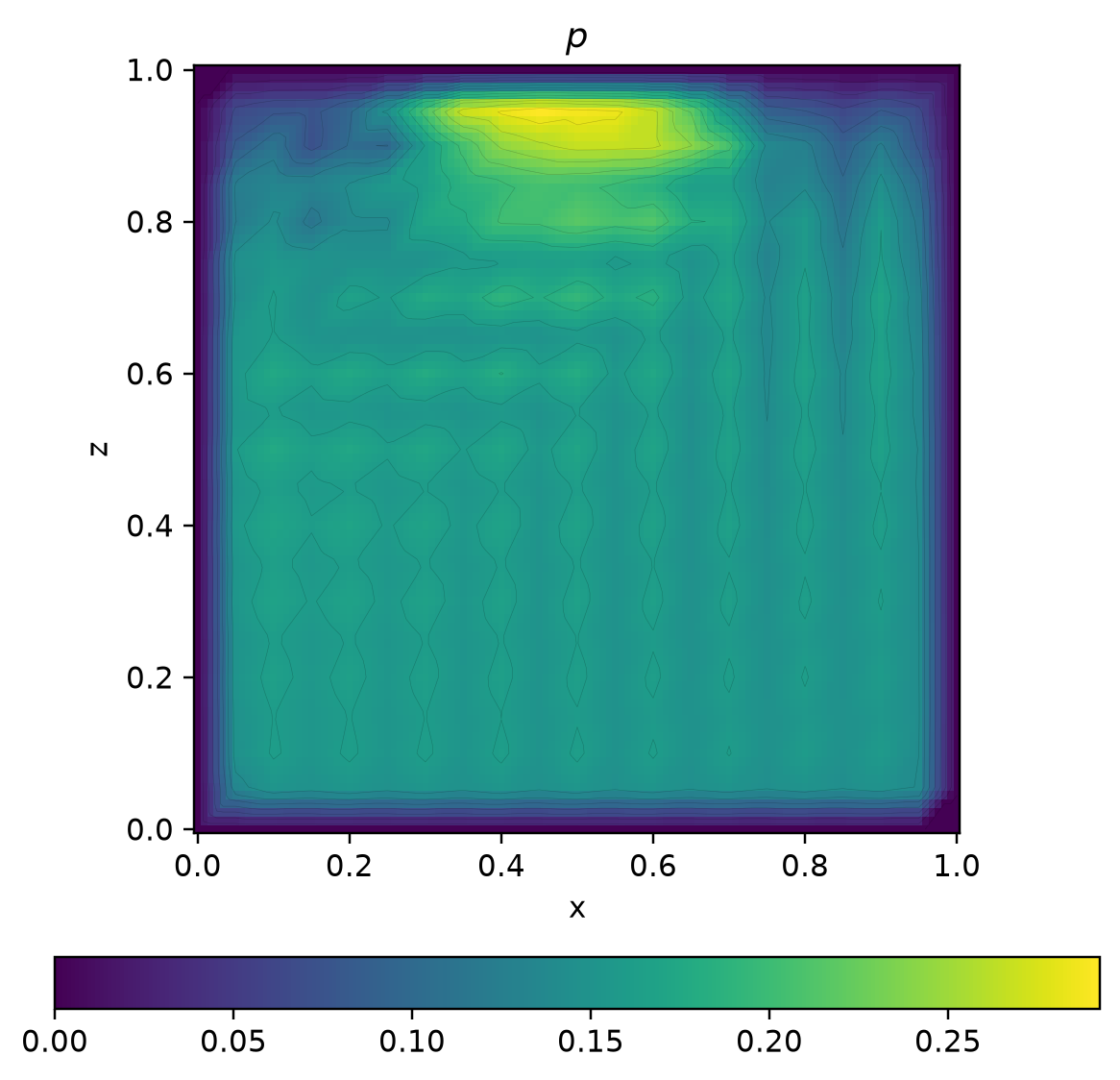}\\[-0.2em]
{\small \(({\bm P}_1,P_1,P_1)\)}
\end{minipage}\hfill
\begin{minipage}{0.24\textwidth}
\centering
\includegraphics[width=\linewidth]{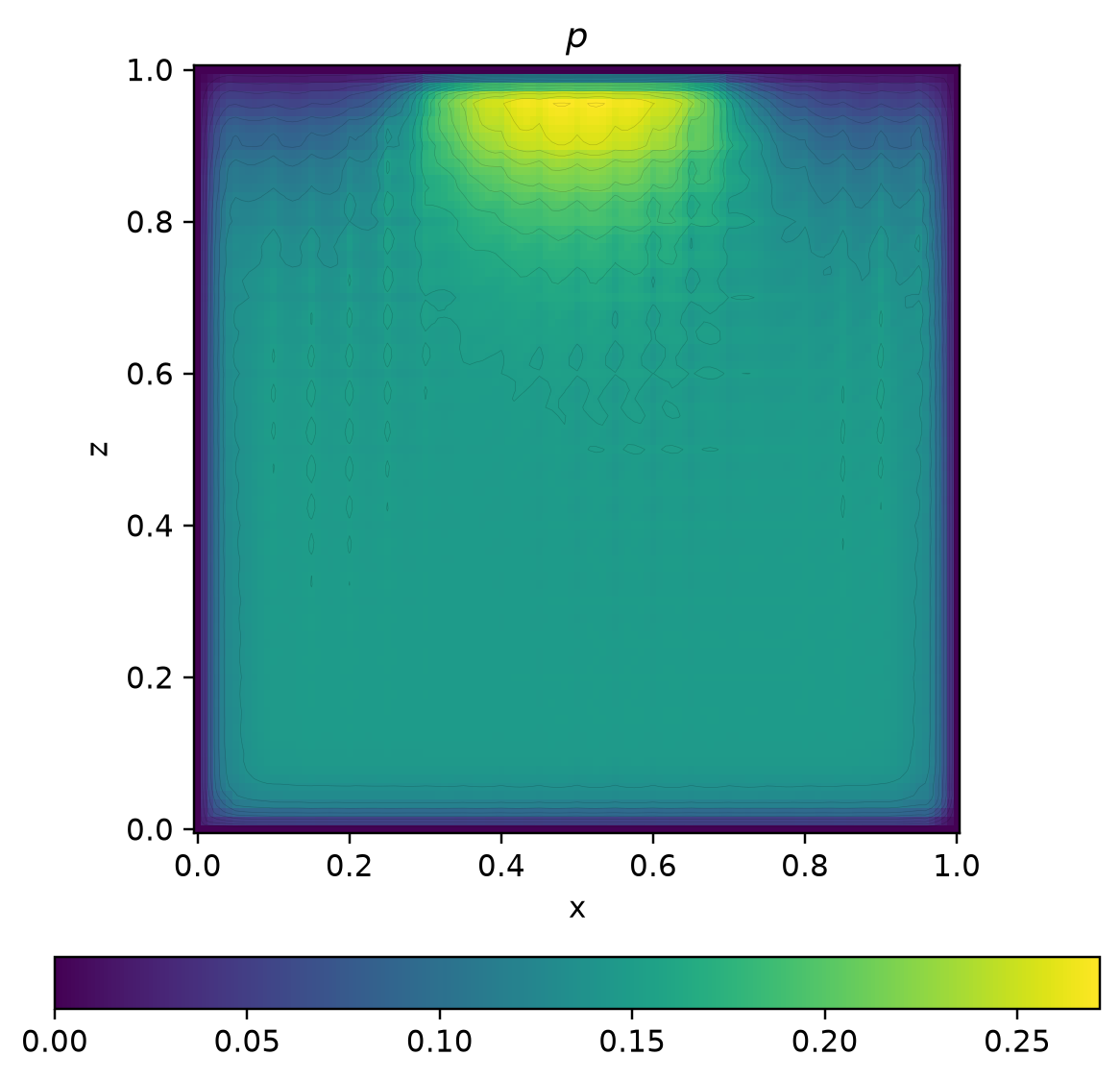}\\[-0.2em]
{\small \(({\bm P}_2,P_2)\)}
\end{minipage}\hfill
\begin{minipage}{0.24\textwidth}
\centering
\includegraphics[width=\linewidth]{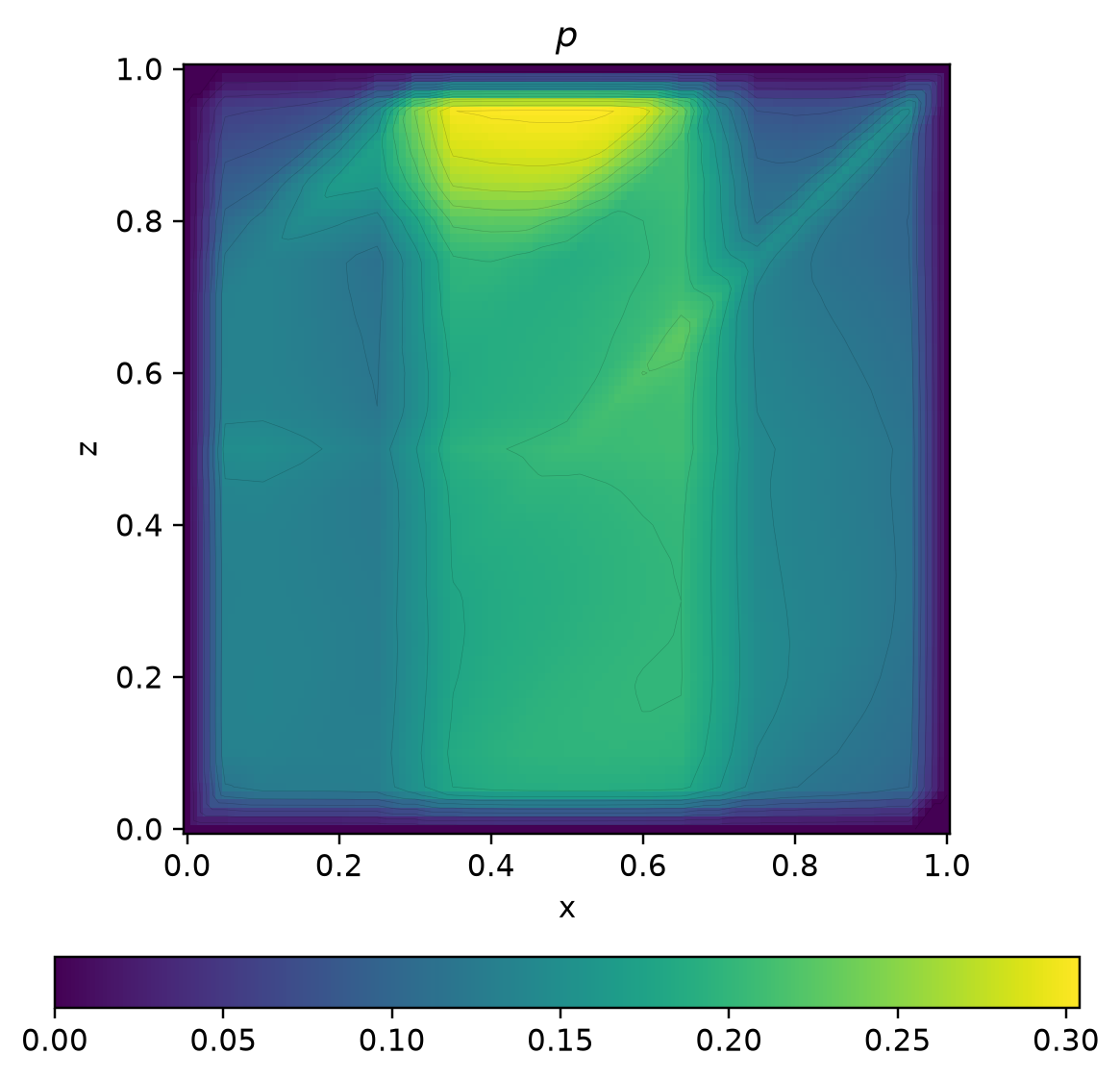}\\[-0.2em]
{\small \(({\bm P}_1,P_1)\)}
\end{minipage}
\caption{Center-section plots of the pressure \(p_h\) on the plane \(y=0.5\) for the footing problem at \(t=0.01\).  The panels compare the five-field formulation with poroelastic spaces \(({\bm P}_2,P_1,P_2)\) and \(({\bm P}_1,P_1,P_1)\) against the four-field formulation with poroelastic spaces \(({\bm P}_2,P_2)\) and \(({\bm P}_1,P_1)\).}
\label{fig:footing-benchmark}
\end{figure}

For the stable five-field computation at \(T_f=0.1\), Figure~\ref{fig:footing-stable-3d} shows the three-dimensional field distributions. The plots illustrate the coupled response of the electromagnetic and poroelastic variables: the electric field is influenced by the pore pressure through the electrokinetic coupling terms, with the strongest response localized near the loaded patch.
\begin{figure}[ht]
\centering
\begin{minipage}[t]{0.32\textwidth}
\centering
\includegraphics[width=\linewidth]{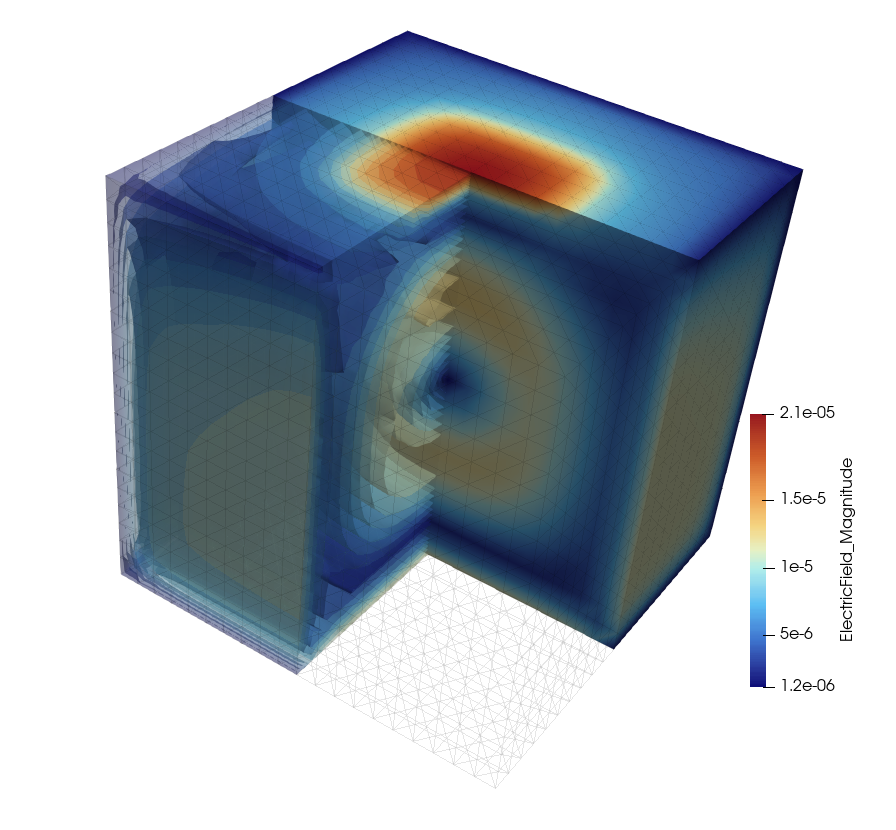}\\[-0.2em]
{\small \(|\E_h|\)}
\end{minipage}\qquad
\begin{minipage}[t]{0.32\textwidth}
\centering
\includegraphics[width=\linewidth]{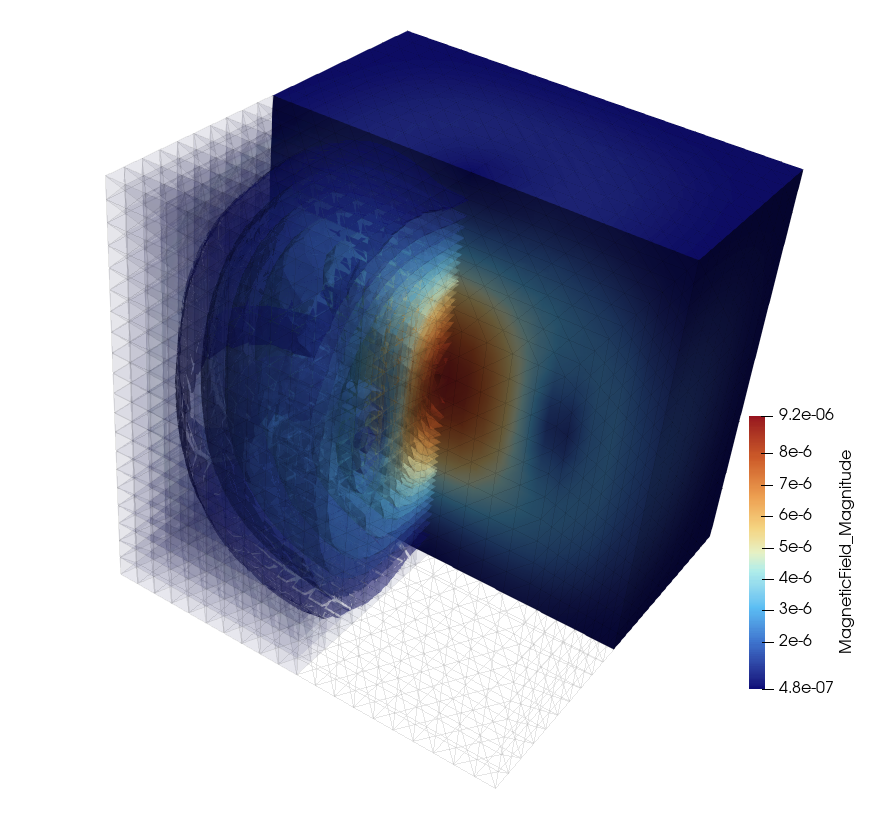}\\[-0.2em]
{\small \(|\HH_h|\)}
\end{minipage}
\vspace{0.6em}
\begin{minipage}[t]{0.32\textwidth}
\centering
\includegraphics[width=\linewidth]{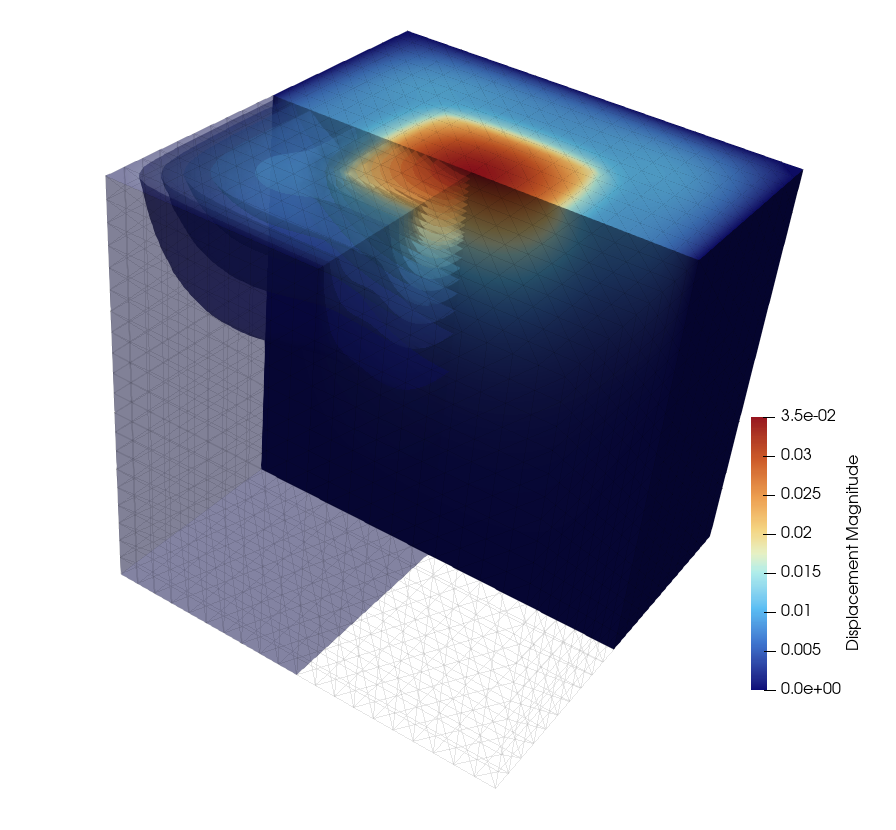}\\[-0.2em]
{\small \(|\uu_h|\)}
\end{minipage}\hfill
\begin{minipage}[t]{0.32\textwidth}
\centering
\includegraphics[width=\linewidth]{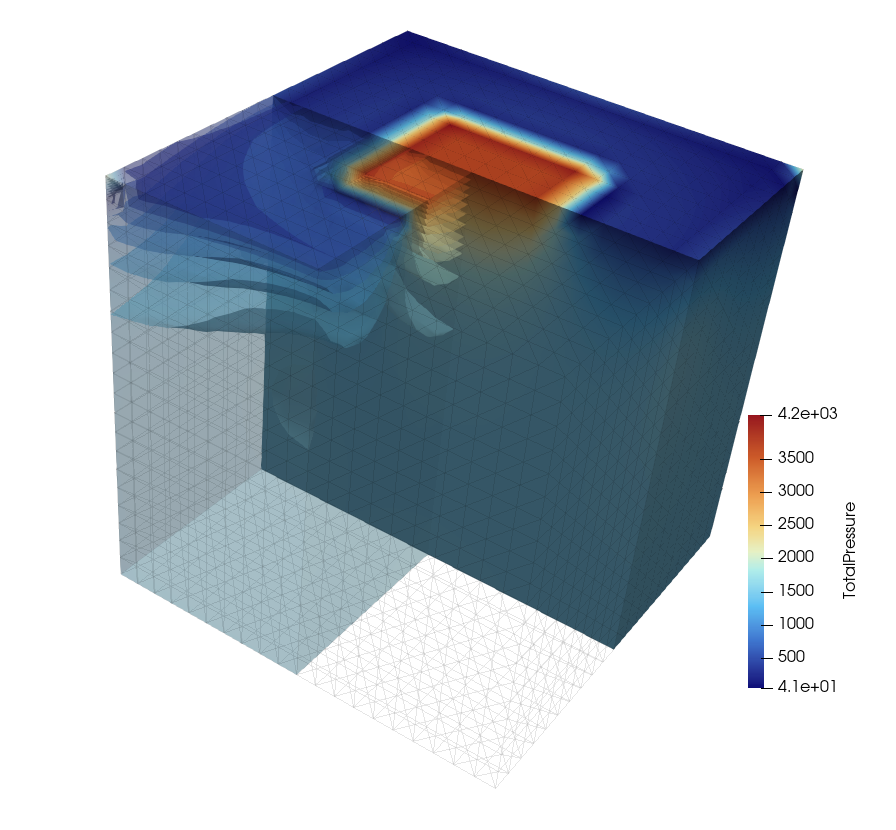}\\[-0.2em]
{\small \(\xi_h\)}
\end{minipage}\hfill
\begin{minipage}[t]{0.32\textwidth}
\centering
\includegraphics[width=\linewidth]{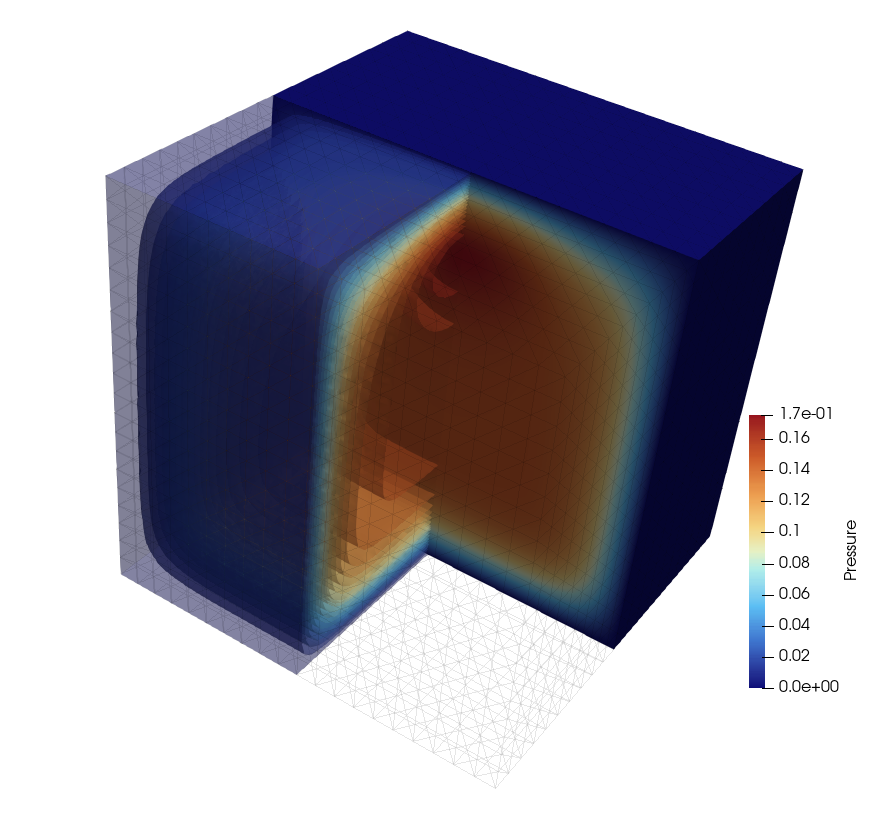}\\[-0.2em]
{\small \(p_h\)}
\end{minipage}
\caption{Three-dimensional visualization of the stable five-field solution at
\(T_f=0.1\).  The top row shows the electromagnetic magnitudes
\(|\E_h|\) and \(|\HH_h|\); the bottom row shows the displacement magnitude
\(|\uu_h|\), total pressure \(\xi_h\), and pore pressure \(p_h\).}
\label{fig:footing-stable-3d}
\end{figure}

\section{Conclusion}
\label{sec:5}
In this work, we introduce a five-field formulation for quasi-static electroporoelasticity and establish the continuous stability estimate under the physical coupling condition \(L<\sqrt{\sigma\kappa}\), providing a stable framework in the nearly incompressible regime.
The resulting BDF2 mixed finite element scheme converges with order \(\mathcal{O}(h^2+\tau^2)\) in the energy norms. 
We further propose an iterative decoupling method that separates the electromagnetic and poroelastic subproblems and converges geometrically to the monolithic solution with a mesh-independent contraction factor. 
We also derive an algebraically equivalent reduced formulation, requiring only one electromagnetic solve per time step, followed by explicit electric-field corrections. 
The numerical results support the theoretical convergence estimates and illustrate the advantage of the stable five-field discretization in suppressing spurious pressure oscillations.

\section*{Acknowledgements}
H. Gu is supported by the National NSF of China No. 123B2016, and Shenzhen University of Information Technology (Grant No. SUIT2025KJ008). M. Cai is supported by the affiliated project award from the Center for Equitable Artificial Intelligence and Machine Learning Systems (CEAMLS) at Morgan State University (project ID 02232301). J. Li was partially supported by the Shenzhen Sci-Tech Fund No. RCJC20200714114556020, Guangdong Basic and Applied Research Fund No. 2023B1515250005, National Center for Applied Mathematics Shenzhen, and SUSTech International Center for Mathematics. 

\bibliographystyle{unsrt}
\bibliography{references}

@article{liu2025splitting,
  author  = {Liu, Xuan and Zou, Yongkui and Zhang, Ran and Cao, Yanzhao and Meir, Amnon J.},
  title   = {Splitting Finite Element Approximations for Quasi-Static Electroporoelasticity Equations},
  journal = {Journal of Scientific Computing},
  volume  = {104},
  number  = {109},
  year    = {2025},
  doi     = {10.1007/s10915-025-03018-5}
}

@article{liu2026multitime,
  author  = {Liu, Xuan and Zou, Yongkui and Meir, Amnon Jacob},
  title   = {Physics-Based Multi-Time Stepping Algorithm for Three-Dimensional Quasi-Static Electroporoelasticity Equations},
  journal = {International Journal of Numerical Analysis and Modeling},
  volume  = {23},
  number  = {2},
  pages   = {252--274},
  year    = {2026},
  doi     = {10.4208/ijnam2026-1011}
}

@article{nedelec1980mixed,
  author  = {N{\'e}d{\'e}lec, Jean-Claude},
  title   = {Mixed Finite Elements in {R}$^3$},
  journal = {Numerische Mathematik},
  volume  = {35},
  pages   = {315--341},
  year    = {1980},
  doi     = {10.1007/BF01396415}
}

@article{biot1962mechanics,
  author  = {Biot, Maurice A.},
  title   = {Mechanics of Deformation and Acoustic Propagation in Porous Media},
  journal = {Journal of Applied Physics},
  volume  = {33},
  number  = {4},
  pages   = {1482--1498},
  year    = {1962},
  doi     = {10.1063/1.1728759}
}

@article{pride1994governing,
  author  = {Pride, Steven R.},
  title   = {Governing Equations for the Coupled Electromagnetics and Acoustics of Porous Media},
  journal = {Physical Review B},
  volume  = {50},
  number  = {21},
  pages   = {15678--15696},
  year    = {1994},
  doi     = {10.1103/PhysRevB.50.15678}
}

@article{haines2006seismoelectric,
  author  = {Haines, Seth S. and Pride, Steven R.},
  title   = {Seismoelectric Numerical Modeling on a Grid},
  journal = {Geophysics},
  volume  = {71},
  number  = {6},
  pages   = {N57--N65},
  year    = {2006},
  doi     = {10.1190/1.2357789}
}

@article{santos2011finite,
  author  = {Santos, Juan E.},
  title   = {Finite Element Approximation of Coupled Seismic and Electromagnetic Waves in Fluid-Saturated Poroviscoelastic Media},
  journal = {Numerical Methods for Partial Differential Equations},
  volume  = {27},
  number  = {2},
  pages   = {351--386},
  year    = {2011},
  doi     = {10.1002/num.20527}
}

@article{santos2012numerical,
  author  = {Santos, Juan E. and Zyserman, Fabio I. and Gauzellino, Patricia M.},
  title   = {Numerical Electro-Seismic Modeling: A Finite Element Approach},
  journal = {Applied Mathematics and Computation},
  volume  = {218},
  number  = {11},
  pages   = {6351--6374},
  year    = {2012},
  doi     = {10.1016/j.amc.2011.12.003}
}

@article{hu2022numerical,
  author  = {Hu, Yu and Meir, Amnon Jacob},
  title   = {Numerical Approximation of Solutions of the Equations of Quasistatic Electroporoelasticity},
  journal = {Numerical Methods for Partial Differential Equations},
  volume  = {38},
  number  = {6},
  pages   = {1929--1947},
  year    = {2022},
  doi     = {10.1002/num.22848}
}

@article{hu2023existence,
  author  = {Hu, Yu and Meir, Amnon Jacob},
  title   = {Existence and Uniqueness of Solutions to the Equations of Quasistatic Electroporoelasticity},
  journal = {Journal of Computational and Applied Mathematics},
  volume  = {431},
  pages   = {115256},
  year    = {2023},
  doi     = {10.1016/j.cam.2023.115256}
}

@book{monk2003finite,
  author    = {Monk, Peter},
  title     = {Finite Element Methods for Maxwell's Equations},
  publisher = {Oxford University Press},
  year      = {2003}
}

@article{kim2011stability,
  author  = {Kim, Jihoon and Tchelepi, Hamdi A. and Juanes, Ruben},
  title   = {Stability and Convergence of Sequential Methods for Coupled Flow and Geomechanics: Fixed-Stress and Fixed-Strain Splits},
  journal = {Computer Methods in Applied Mechanics and Engineering},
  volume  = {200},
  number  = {13--16},
  pages   = {1591--1606},
  year    = {2011},
  doi     = {10.1016/j.cma.2010.12.022}
}

@article{mikelic2013convergence,
  author  = {Mikeli{\'c}, Andro and Wheeler, Mary F.},
  title   = {Convergence of Iterative Coupling for Coupled Flow and Geomechanics},
  journal = {Computational Geosciences},
  volume  = {17},
  pages   = {455--461},
  year    = {2013},
  doi     = {10.1007/s10596-012-9318-y}
}

@article{both2017robust,
  author  = {Both, Jakub W. and Borregales, Mauricio and Nordbotten, Jan M. and Kumar, Kundan and Radu, Florin A.},
  title   = {Robust Fixed Stress Splitting for {Biot}'s Equations in Heterogeneous Media},
  journal = {Applied Mathematics Letters},
  volume  = {68},
  pages   = {101--108},
  year    = {2017},
  doi     = {10.1016/j.aml.2016.12.019}
}

@article{lee2016robust,
  author  = {Lee, Jeonghun J.},
  title   = {Robust Error Analysis of Coupled Mixed Methods for {Biot}'s Consolidation Model},
  journal = {Journal of Scientific Computing},
  volume  = {69},
  number  = {2},
  pages   = {610--632},
  year    = {2016},
  doi     = {10.1007/s10915-016-0210-0}
}

@article{burger2021virtual,
  author  = {B{\"u}rger, Raimund and Kumar, Sarvesh and Mora, David and Ruiz-Baier, Ricardo and Verma, Nitesh},
  title   = {Virtual Element Methods for the Three-Field Formulation of Time-Dependent Linear Poroelasticity},
  journal = {Advances in Computational Mathematics},
  volume  = {47},
  number  = {1},
  pages   = {2},
  year    = {2021},
  doi     = {10.1007/s10444-020-09826-7}
}

@article{gu2023iterative,
  author  = {Gu, Huipeng and Cai, Mingchao and Li, Jingzhi},
  title   = {An Iterative Decoupled Algorithm with Unconditional Stability for {Biot} Model},
  journal = {Mathematics of Computation},
  volume  = {92},
  number  = {341},
  pages   = {1087--1108},
  year    = {2023},
  doi     = {10.1090/mcom/3809}
}

@article{cai2023some,
  author  = {Cai, Mingchao and Gu, Huipeng and Li, Jingzhi and Mu, Mo},
  title   = {Some Optimally Convergent Algorithms for Decoupling the Computation of {Biot}'s Model},
  journal = {Journal of Scientific Computing},
  volume  = {97},
  number  = {2},
  pages   = {48},
  year    = {2023},
  doi     = {10.1007/s10915-023-02365-5}
}

@article{oyarzua2016locking,
  author  = {Oyarz{\'u}a, Ricardo and Ruiz-Baier, Ricardo},
  title   = {Locking-Free Finite Element Methods for Poroelasticity},
  journal = {SIAM Journal on Numerical Analysis},
  volume  = {54},
  number  = {5},
  pages   = {2951--2973},
  year    = {2016},
  doi     = {10.1137/15M1050082}
}

@misc{multiphenicsx,
  author = {Ballarin, Francesco},
  title  = {{multiphenicsx}: Easy Prototyping of Multiphysics Problems in {FEniCSx}},
  year   = {2020},
  url    = {https://multiphenics.github.io/}
}

@misc{BarattaEtal2023,
  author       = {Baratta, Igor A. and Dean, Joseph P. and Dokken, J{\o}rgen S. and Habera, Michal and Hale, Jack S. and Richardson, Chris N. and Rognes, Marie E. and Scroggs, Matthew W. and Sime, Nathan and Wells, Garth N.},
  title        = {{DOLFINx}: The Next Generation {FEniCS} Problem Solving Environment},
  year         = {2023},
  howpublished = {preprint},
  doi          = {10.5281/zenodo.10447666}
}

@article{mcghee2011class,
  author  = {McGhee, Des and Picard, Rainer},
  title   = {A Class of Evolutionary Operators and Its Applications to Electroseismic Waves in Anisotropic, Inhomogeneous Media},
  journal = {Operators and Matrices},
  volume  = {5},
  number  = {4},
  pages   = {665--678},
  year    = {2011},
  doi     = {10.7153/oam-05-48}
}

@article{mcghee2011electroseismic,
  author  = {McGhee, Des and Picard, Rainer},
  title   = {On Electroseismic Waves in Anisotropic, Inhomogeneous Media},
  journal = {GAMM-Mitteilungen},
  volume  = {34},
  number  = {1},
  pages   = {76--83},
  year    = {2011},
  doi     = {10.1002/gamm.201110012}
}

@article{altmann2024semi,
  author  = {Altmann, Robert and Maier, Robert and Unger, Benjamin},
  title   = {Semi-explicit Integration of Second Order for Weakly Coupled Poroelasticity},
  journal = {BIT Numerical Mathematics},
  volume  = {64},
  number  = {20},
  year    = {2024},
  doi     = {10.1007/s10543-024-01021-0},
  url     = {https://doi.org/10.1007/s10543-024-01021-0}
}

@article{brenner1993nonconforming,
  author  = {Brenner, Susanne C.},
  title   = {A Nonconforming Mixed Multigrid Method for the Pure Displacement Problem in Planar Linear Elasticity},
  journal = {SIAM Journal on Numerical Analysis},
  volume  = {30},
  number  = {1},
  pages   = {116--135},
  year    = {1993}
}

\end{document}